\documentclass{amsart}

\usepackage[letterpaper,top=2cm,bottom=2cm,left=3cm,right=3cm,marginparwidth=1.75cm]{geometry}

\usepackage[utf8]{inputenc}
\usepackage{amsmath}
\usepackage{mathtools}   
\usepackage{amsfonts}
\usepackage{listings}
\usepackage{cancel}
\usepackage{comment}
\usepackage{mathabx}
\usepackage{amssymb}
\usepackage{tikz}
\usepackage{subfig}
\usepackage{float}
\usetikzlibrary{arrows, automata, decorations.markings}
\usepackage{hyperref}
\usepackage{verbatim}

\usepackage{todonotes}

\theoremstyle{plain}   						
\newtheorem{theorem}{Theorem}[section]			
\newtheorem{lemma}[theorem]{Lemma} 		 	
\newtheorem{proposition}[theorem]{Proposition}

\DeclareMathOperator{\sech}{sech}
\theoremstyle{remark} 
\newtheorem{remark}{Remark}[section]

\theoremstyle{definition}
\newtheorem{definition}{Definition}[section]
\numberwithin{equation}{section}

\title[]{Confinement and orbital stability of solitons \\ of the NLS equation
on metric graphs}

\author[]{Martino Caliaro}
\address{Gran Sasso Science Institute, L'Aquila, Italy}
\email{martino.caliaro@gssi.it}

\author[]{Diego Noja}
\address{Dipartimento di Matematica e Applicazioni, Universit\`a
 di Milano Bicocca,  via R. Cozzi 55, 20126, Milano, Italy}
\email{diego.noja@unimib.it} 

\thanks{Corresponding author: Diego Noja, e-mail: diego.noja@unimib.it}

\date{\today}

\begin{document}

\begin{abstract}

We study the behavior of untrapped soliton states for the subcritical, time-dependent focusing NLS equation on a large family of non-compact metric graphs with Kirchhoff boundary conditions. We give first our results for metric graphs characterized by a topological assumption (``Assumption~H'') which rules out the existence of a ground state for all members of the class, with a single exception: the bubble-tower metric graph. We present two main results.
First, we show that if the initial datum is close (in the energy norm) to a soliton placed on a single half-line of the graph and sufficiently far from the nearest vertex, then the corresponding solution remains confined to the same half-line for all times, and close to the soliton, up to a remainder small in the energy norm. As a nontrivial application, this yields reflection of a slow soliton upon collision with the compact core of the graph, a phenomenon that first we prove and then we further investigate numerically.
Second, for the exceptional case of bubble-tower graphs, we prove that the existing ground state is orbitally stable. We emphasize that this example does not allow a direct application of the Cazenave--Lions orbital stability argument. Finally, we show, with several examples, how the ideas and methods developed here extend beyond the class of metric graphs satisfying Assumption~H, which turns out to be an instance of a more general energy threshold phenomenon. In particular, we treat the meaningful case of the line in the presence of a smooth potential or a delta interaction.

\end{abstract}

\maketitle

\begin{footnotesize}
 \end{footnotesize}

\section{Introduction}

The study of the NLS equation on metric graphs has experienced substantial growth over the last decade. A rigorous analysis was probably initiated by \cite{ACFN11}, where the behavior of fast solitons of the cubic NLS equation after collision with the vertex of a three-edge star graph under several boundary conditions was investigated, in analogy with the seminal study on the line in \cite{HMZ07}. 
In the subsequent years, research in the area has focused almost exclusively on the stationary NLS equation, describing the amplitude of bound states, i.e. periodic localized solutions of the NLS equation. This has led to a large body of results on the existence and properties of ground states (corresponding to several definitions and obtained by several methods) and on the existence of various classes of standing waves. In some cases, their orbital stability and instability has been established, essentially as a by-product of the variational or spectral analysis of the stationary problem and its linearization. These results rely on the methods originating from the seminal works of Cazenave and Lions \cite{cazenave_lions} and of Grillakis, Shatah and Strauss \cite{Grillakis88, GSS87}, together with later developments (see, e.g., \cite{KNP22} and references therein for a review). 

In this paper we firstly provide results on the time behavior of untrapped, soliton-like solutions of the NLS equation on a broad and well-defined class of non-compact metric graphs with finitely many edges, whose importance was highlighted in \cite{adami_2}, \cite{adami}. The graphs we consider are characterized by the property that every point lies on a trail (i.e., a sequence of distinct edges) containing two half-lines. This topological assumption is called Assumption~H in \cite{adami_2},\cite{adami}; it is shown there that Assumption~H implies the generic absence of an NLS ground state (defined as a minimizer of the NLS energy at fixed $L^2$ mass), with a single exception: the so-called \emph{tower of bubbles}, for which the ground state exists and can be exhibited explicitly. This class is therefore large enough to contain both the generic no-ground-state regime and the unique exceptional case with a ground state and it offers a natural and robust setting. At the same time, as we will discuss in the last section, the confinement argument can be extended well beyond Assumption~H and it reveals a more fundamental threshold phenomenon. We now come to the description of our results.

\par
First, in Proposition~\ref{main_prop_1}, without distinguishing between the two cases, we show that if an initial datum is supported on a single half-line of the graph and is close in the energy norm to a truncated soliton, then along the NLS flow it remains confined to the same half-line and stays close (up to a suitable rephasing and translation) to the truncated soliton for all times.
This statement resembles orbital stability, except that the standard definition compares solutions with an actual solitary/standing wave, whereas here the solution is compared with a reference function (a soliton on the half-line plus a small error on the rest of the graph), to which no genuine solitary wave corresponds. The idea is that a soliton placed on a half-line sufficiently far from the compact core (so that the truncation error is small in the energy norm) still effectively behaves as a configuration that is stable for the NLS dynamics, in an appropriate and rigorous sense.  A further piece of information contained in Proposition~\ref{main_prop_1} is that the initial datum can be chosen so that the corresponding soliton-like solution never comes closer than a fixed \emph{a priori} distance to the compact core of the metric graph (see Proposition~\ref{main_prop_1} for the precise formulation).

For comparison, a related result was recently proved for the KdV equation on the half-line with Dirichlet boundary condition (see \cite{CavMun19} and the subsequent \cite{CavMun23, CCCav26} the latter treating also different b.c.). Note that a half-line is a very special metric graph, and the Dirichlet condition presumably prevents the existence of an exact solitary wave, as remarked by the authors of those papers. Their strategy uses an adaptation of the ideas in \cite{MMT02}, based on the introduction of certain almost-conserved quantities and their monotonicity properties. Here, instead, we consider the NLS equation in the substantially more general setting of metric graphs. Our proofs exploit a version of concentration--compactness adapted to non-compact metric graphs with finitely many edges (in particular, covering all graphs satisfying Assumption~H), developed in \cite{CFN17}, combined with a contradiction argument. Roughly speaking, if the claim were false along a sequence of initial data, the corresponding evolutions would produce a minimizing sequence. Such a minimizing sequence could either converge (which is ruled out by the absence of ground states) or be \emph{runaway}, i.e.\ escape to infinity along a single half-line carrying all the mass. The runaway scenario is also excluded, since it would imply a false inequality, yielding a contradiction. In the exceptional tower of bubbles case, where mass and energy alone are not sufficient, we introduce an additional functional $F$ (see \eqref{F}) that is almost conserved along minimizing sequences; this allows us to conclude, again by contradiction.

A relevant example of initial datum satisfying the hypotheses of Proposition~\ref{main_prop_1} is that of a slow soliton placed far from the vertex and moving toward it, discussed in Subsection~\ref{SlowS}. In this case, the main result implies that the solution remains on the same half-line and stays close to a soliton in the energy norm, if the velocity initially present is below a critical velocity: slow solitons are reflected, as shown in Proposition \ref{prop:slow-solitons}. This behavior is reminiscent of the so-called \emph{quantum reflection} of matter-wave solitons, well known in the physics of Bose--Einstein condensates and matter wave solitons (see, e.g., \cite{Cornish08,Brand06, ErnstBrand2010}). Beyond the rigorous result above, the quantum-reflection phenomenon is further investigated numerically in Section~\ref{conclusions}. In particular, we show that the kinetic energy of the incoming soliton attains its maximum at the collision time, a markedly non-classical behavior.\\
We also stress that analogous results on confinement, orbital stability and slow-soliton reflection hold for the more standard case of the subcritical NLS on the line with a repulsive potential; this is proved in the final part of the paper (see Proposition~\ref{Line}), where we also compare our approach with the existing literature on the subject. This model, and the other discussed in Section~6, show that confinement and reflection of slow solitons have to be interpreted as a general threshold energy phenomenon and they are not restricted to metric graphs in the class $H$. We find that the latter remains however a convenient setting, useful to exhibit the relevant behaviors in both the situations of absence or presence of ground state.
\par
Our second main result concerns the tower of bubbles (see the definition in Subsection~\ref{BBL} and Figure~\ref{fig:1}). This is the only subfamily within class~H that admits an NLS ground state. The general expectation is that ground states are orbitally stable, following the celebrated argument of Cazenave and Lions \cite{cazenave_lions}. Here, however, one must take into account that a key assumption in \cite{cazenave_lions} is the relative compactness of all minimizing sequences in the energy space. This fails for tower of bubbles graphs, because there exist minimizing sequences that escape along a single half-line (runaway sequences in the terminology above; see \cite{CFN17} and Subsection~\ref{sect: concentration_comp}), preventing convergence. To overcome this difficulty and prove orbital stability of the ground state (see Proposition~\ref{main_prop_2}), we again exploit the functional $F$ introduced above, showing that instability of the ground state would contradict the continuity of the map $t\mapsto F(\xi(t))$ for $t\in\mathbf{R}$, where $\xi$ denotes the solution of the time-dependent NLS equation.

\par\noindent
The paper is organized as follows. In Section~2 we collect preliminary notation and background on metric graphs, introduce several constructions used later, and recall variational properties of minimizing sequences and of the NLS equation on metric graphs. Most of this material is known but scattered in the literature; here we present it in a compact and unified way. In Section~3 we focus on the class of non-compact metric graphs under consideration, including its definition and basic properties. In Section~4 we prove the first main result on confinement of solitons, stated as Proposition~\ref{main_prop_1}; the dynamics of slow solitons is also discussed there, see Proposition \ref{prop:slow-solitons}. In Section~5 we prove orbital stability of ground states for tower-of-bubbles graphs, stated as Proposition~\ref{main_prop_2}. In Section~6 we discuss how the hypotheses of the presented results can be weakened in various directions and we provide some open problems. Finally, we present numerical simulations of collision and reflection of a slow soliton at the vertex of a three-edge star graph, illustrating the results of Section~4. The simulations use the package QGLAB recently developed in \cite{goodman}.

\section{Preliminaries}
\subsection{Quantum graphs}
We begin by introducing some tools which are necessary to deal with the NLS equation on metric graphs.\\ 
We consider a metric graph $\mathcal{G} = (V,E)$, where $V$ is the set of vertices and $E$ is the set of edges. We will always consider finite graphs, so that the cardinalities $|V|$ and $|E|$ are finite. Each edge $e \in E$ has an associated length $L_e \in (0,+\infty]$ and it is identified with the interval $I_e = [0,L_e]$, if $L_e$ is finite, or with the interval $I_e=[0,+\infty)$, if $L_e$ is infinite. An edge of finite length is called bounded edge, while an edge of infinite length is called half-line. If the graph $\mathcal{G}$ possesses a half-line, we say that $\mathcal{G}$ is non-compact; we say $\mathcal{G}$ is compact if it has no half-lines. To every bounded edge we associate two vertices, and to every half-line we associate a single vertex. On each interval $I_e$ we fix a coordinate $x$, such that the points $x=0$ and $x = L_e$ correspond to vertices if $L_e< +\infty$; if instead $L_e = +\infty$, the point $x = 0$ corresponds to the unique vertex associated to the half-line $e$. When a vertex $\underline{v} \in V$ belongs to an edge $e \in E$, we write $\underline{v} \in e$. We denote by $\{e \prec \underline{v}\}$ the set of edges connecting the vertex $\underline{v} \in V$. In the following, we will denote points on the graph with $\underline{x} =(e,x)$, where $e \in E$ identifies the edge and $x \in I_e$ the coordinate on the corresponding edge. Given two points $\underline{x}$ and $\underline{y}$ on $\mathcal{G}$, we define the distance $d(\underline{x},\underline{y})$ as infimum of the length of the paths connecting the two points. With this distance, the couple $(\mathcal{G},d)$ is a locally compact metric space. We call \textit{trail} a finite sequence of consecutive edges $e \in E$ in which no edge is repeated (a vertex, instead, may be repeated, for example in presence of loops). A \textit{path} is a trail in which no vertex is repeated.  We say that a graph is connected if for every pair of vertices $\underline{v_1},\underline{v_2} \in V$, there exists a path connecting the two. If a graph is not connected, we say it is disconnected.\\
A function $\psi: \mathcal{G} \to \mathbf{C}$ can be regarded as a family of functions $\{\psi_e\}_{e \in E}$ with $\psi_e: I_e \to \mathbf{C}$. In particular, if $\underline{x} = (e,x)$, then
\begin{equation}
    \psi(\underline{x}) = \psi_e(x).
\end{equation}
At convenience, we will also use the notation $(\psi)_e = \psi_e$ for $e \in E$.
The spaces $L^q(\mathcal{G})$, $1\leq q\leq \infty$, are made of functions $\psi$ on $\mathcal{G}$ such that $\psi_e \in L^q(I_e)$ for every $e \in E$ and we define
\begin{equation}
    ||\psi||^q_{L^q(\mathcal{G})} = \sum_{e \in E} ||\psi_e||^q_{L^q(I_e)}, \ 1 \leq q < \infty \quad \text{and} \quad ||\psi||_{L^{\infty}(\mathcal{G})} = \max_{e \in E} ||\psi_e||_{L^{\infty}(I_e)}
\end{equation}
The space $L^2(\mathcal{G})$ forms a Hilbert space, with the scalar product defined in the natural way.
We denote by $C(\mathcal{G})$ the set of functions which are continuous on $\mathcal{G}$, and we introduce the spaces
\begin{equation}
    H^1(\mathcal{G}) = \{ \psi \in C(\mathcal{G}) \ \text{s.t.} \ \psi_e \in H^1(I_e), \forall e \in E\}
\end{equation}
equipped with the norm
\begin{equation}
    ||\psi||^2_{H^1(\mathcal{G})} = \sum_{e \in E}||\psi_e||^2_{H^1(I_e)}
\end{equation}
and 
\begin{equation}
    H^2(\mathcal{G}) = \{ \psi \in C(\mathcal{G}) \ \text{s.t.} \ \psi_e \in H^2(I_e), \forall e \in E\}
\end{equation}
equipped with the norm
\begin{equation}
    ||\psi||^2_{H^2(\mathcal{G})} = \sum_{e \in E}||\psi_e||^2_{H^2(I_e)}.
\end{equation}
We recall the Gagliardo-Nirenberg inequality for metric graphs \cite{adami}, which we shall use in the case of connected non-compact graphs. Notice that the constants $C,C'>0$ appearing in the statement do not depend on the structure of the non-compact graph.
\begin{proposition}[Gagliardo-Nirenberg inequality]
    Let $\mathcal{G}$ be a connected non-compact graph. If $q\in [2,+\infty)$, there exists $C>0$ such that
    \begin{equation}
        ||\psi||^q_{L^q(\mathcal{G})} \leq C ||\psi'||^{\frac{q}{2}-1}_{L^2(\mathcal{G})}||\psi||^{\frac{q}{2}+1}_{L^2(\mathcal{G})}
        \label{gagliardo-nirenberg}
    \end{equation}
    for any $\psi \in H^1(\mathcal{G})$.  Moreover, there exists $C'>0$ such that
\begin{equation}
||\psi||^2_{L^{\infty}(\mathcal{G})} \leq C' ||\psi'||_{L^2(\mathcal{G})}||\psi||_{L^2(\mathcal{G})},
\end{equation}
for any $\psi \in H^1(\mathcal{G})$.
\end{proposition}
\subsection{Sub-graphs}
\label{sub-graphs}
In the following we will consider the restriction of a function $u \in H^1(\mathcal{G})$ to a sub-graph of $\mathcal{G}$. A sub-graph $\mathcal{B} = (E',V')$ will be formed from a subset of the vertices and edges of $\mathcal{G}$, and will be possibly disconnected. The vertex subset $V'$ will include only the vertices connected to an edge of the subset $E'$. For $u \in H^1(\mathcal{G})$, we use the notation
\begin{equation}
    ||u||^2_{H^1(\mathcal{B})} = \sum_{e \in E'} ||u||^2_{H^1(I_e)},
\end{equation}
and similarly for the other $L^q$-norms, with $1\leq q<\infty$.
We also make use of the following construction. Suppose $\mathcal{G} = (V,E)$ is a connected graph and $\mathcal{T} =(e_1, e_2,  ..., e_n)$ is a trail of $\mathcal{G}$, with $n \in \mathbf{N}$. Suppose that $e_1$ and $e_n$ are half-lines. Call $\underline{v_1}$ the vertex associated to $e_1$. Given $u \in H^1(\mathcal{G})$, our goal is to define a continuous function $u|_{\mathbf{R}} \in H^1(\mathbf{R})$, by restricting  $u$ to the trail $\mathcal{T}$. We proceed as follows. To the edge $e_1$ of the trail, we associate the interval $J_{1} = (-\infty,0]$. To the edge $e_2$ we associate the interval $J_2 = [0,L_{e_2}]$, where $L_{e_2}\geq 0$ is the length of $e_2$. We proceed by consecutive edges in this way, until the entire trail has been covered, except for the half-line $e_n$. We associate to $e_n$ the interval $J_{n} = [l,+\infty]$, where $l \geq 0$ is the total length of the edges $e_2,...,e_{n-1}$. If $u \in H^1(\mathcal{G})$, we define the function $u|_{\mathbf{R}}: \mathbf{R} \to \mathbf{C}$ as follows. For $x \in J_{1}$, we set $u|_{\mathbf{R}}(x) = u_{e_1}(-x)$. If the edge $e_2$ is outgoing from $\underline{v_1}$, we set $u|_{\mathbf{R}}(x) = u_{e_2}(x)$, for $x \in J_{e_2}$. If the edge $e_2$ is incoming in $\underline{v_1}$, we set $u|_{\mathbf{R}}(x) = u_{e_2}(L_{e_2}-x)$. We proceed in a similar way, until we reach the half-line $e_n$. For $x \in J_{n}$, we set  $u|_{\mathbf{R}}(x) = u_{e_n}(x-l)$. We obtain in this way a continuous function $u|_{\mathbf{R}} \in H^1(\mathbf{R})$ such that 
\begin{equation}
    ||u|_{\mathbf{R}}||^2_{H^1(\mathbf{R})} = \sum_{e \in \mathcal{T}}||u_e||^2_{H^1(I_e)},
\end{equation}
and that 
\begin{equation}
    ||u|_{\mathbf{R}}||^q_{L^q(\mathbf{R})} = \sum_{e \in \mathcal{T}}||u_e||^q_{L^q(I_e)}
\end{equation}
for any $1 \leq q < \infty$.
If $\mathcal{B}_{\mathcal{T}}$ represents the graph obtained by removing from $\mathcal{G}$ the edges in the trail $\mathcal{T}$, and the resulting disconnected vertices, we write
\begin{equation}
    ||u||^2_{H^1(\mathcal{G})} = ||u||^2_{H^1(\mathcal{B}_{\mathcal{T}})} + ||u|_{\mathbf{R}}||^2_{H^1(\mathbf{R})},  
\end{equation}
and similarly for the other $L^q$-norms, with $1 \leq q < \infty$.\\
\subsection{NLS equation on graphs}
In this section we recall some properties of the NLS equation on metric graphs. For a more complete treatment we refer to \cite{CFN17}.\\
The NLS equation on a graph $\mathcal{G} = (V,E)$ reads 
\begin{equation}
    i\partial_t\psi = H\psi -|\psi|^{p-2}\psi
    \label{NLS_eq}
\end{equation}
In the following, the power $p$ will be a given number in the interval $(2,6)$, and we refer to this regime as \textit{subcritical}. The operator $H$ in (\ref{NLS_eq}) acts as $$(H \psi)_e = -\psi''_e, \qquad \forall e \in E.$$
The domain of $H$ is defined by means of the \textit{Kirchhoff condition} at the vertices:
\begin{equation}
    D(H) = \{\psi \in H^2(\mathcal{G}), \ \text{s.t.} \ \sum_{e \prec \underline{v}}\partial_o\psi_e(\underline{v}) = 0 \quad \forall \underline{v}  \in V\},
    \label{domain}
\end{equation}
where $\partial_o$ denotes the outgoing derivative from the vertex, and it corresponds to $\partial_x$ or to $-\partial_x$ according to the orientation of the edge. In particular, the operator $(H,D(H))$ is self-adjoint on $L^2(\mathcal{G})$ and the quadratic form associated with $H$ is 
\begin{equation}
    Q[\psi] = \frac{1}{2}||\psi'||^2_{L^2(\mathcal{G})}, \quad \text{with domain} \quad  \mathcal{D}(Q) = H^1(\mathcal{G}).
\end{equation}
Finally, the nonlinearity in equation (\ref{NLS_eq}) is interpreted edge by edge, i.e. $(|\psi|^{p-2}\psi)_e = |\psi_e|^{p-2}\psi_e, \ \forall e \in E. $\\ The equation in (\ref{NLS_eq}) admits two conserved quantities, namely the mass
\begin{equation}
    M(\psi) = ||\psi||^2_{L^2(\mathcal{G})}
    \label{mass}
\end{equation}
and the energy
\begin{equation}
    E(\psi,\mathcal{G}) = \frac{1}{2} ||\psi'||_{L^2(\mathcal{G})}^2 - \frac{1}{p}||\psi||^p_{L^p(\mathcal{G})}
    \label{energy}
\end{equation}
The well-posedness of the Cauchy problem associated to equation (\ref{NLS_eq}) has been proven in  \cite{CFN17}, by applying a Banach fixed point theorem. The result can be summarized as follows.
\begin{proposition}
    Set $2< p < 6$. For any $\psi_0 \in H^1(\mathcal{G})$, there exists a unique global solution $\psi \in C^0(\mathbf{R}, H^1(\mathcal{G})) \cap C^1(\mathbf{R}, H^1(\mathcal{G})^{\star})$ to equation (\ref{NLS_eq}). Moreover,
    \begin{equation}
        E(\psi(t),\mathcal{G}) = E(\psi_0,\mathcal{G}) \quad \text{ and} \quad M(\psi(t)) = M(\psi_0) \qquad \forall t \in \mathbf{R}\ .
    \end{equation}
    \label{prop: cauchy}
\end{proposition}
\subsection{Ground states on non-compact graphs}
Given a non-compact graph $\mathcal{G}$ and a mass $\mu >0$, we introduce the following set of functions with prescribed mass
\begin{equation}
    H^1_{\mu}(\mathcal{G}) = \{ \psi \in H^1(\mathcal{G}), \quad M(\psi) = \mu\}.
\label{eq: H1_mu}
\end{equation}
We then consider the following minimization problem
\begin{equation}
    \mathcal{E}_{\mathcal{G}}(\mu) := \inf_{\psi \in H^1_{\mu}(\mathcal{G})}E(\psi,\mathcal{G}).
    \label{minimization_problem}
\end{equation} 
In the subcritical regime, the Gagliardo-Nirenberg inequality (\ref{gagliardo-nirenberg}) implies that, for any $\mu >0$, the energy in (\ref{energy}) is bounded from below in the set $H^1_{\mu}(\mathcal{G})$, i.e. $  \mathcal{E}_{\mathcal{G}}(\mu)>  -\infty $ for any $\mu >0$. We give the following definition.
\begin{definition}
    Given $\mathcal{G}$ and $\mu >0$, a function $u \in H_{\mu}^1(\mathcal{G})$ is called a ground state if and only if
    \begin{equation}
        E(u, \mathcal{G}) = \mathcal{E}_{\mathcal{G}}(\mu).
    \end{equation}
\end{definition}
A familiar example of metric graph, trivial but however of great importance in the following, is the real line $\mathcal{G}= \mathbf{R}$. It can be considered as composed by two half-lines joined at a vertex, where the Kirchhoff condition reduces to continuity of the function and its derivative. In this case, the set of ground states is well known and explicit: given a mass $\mu >0$ the set of ground states is the set
\begin{equation}\label{Omu}
    \mathcal{O}_{\mu} = \{ e^{i\theta}\phi_{\mu}(\cdot-x_0), \  \text{with} \  (\theta, x_0) \in \mathbf{R}^2\}
\end{equation}
where
\begin{equation}
    \phi_{\mu}(x) = C_p \ \mu^{\frac{2}{6-p}}\sech^{\frac{2}{p-2}}(c_p \mu^{\frac{p-2}{6-p}}x),
    \label{eq: soliton}
\end{equation}
 with the constants $C_p$ and $c_p$ depending only on $p$ (see \cite{sulem}). The function $\phi_{\mu}$ is the unique positive and even ground state.\\
In the case of connected non-compact graphs, NLS ground states do not always exist. This fact has been investigated in \cite{adami_2,adami_rev}, and it turns out that the existence of ground states is related to the topology of graphs, as we will see in Section \ref{H}. When they exist, ground states are real and positive. For now, we state the following proposition from \cite{adami_2}, of which we report the proof.
 \begin{proposition}
     Suppose $\mathcal{G}$ is a connected non-compact graph and let $\mu >0$. Then
     \begin{equation}
         \mathcal{E}_{\mathcal{G}}(\mu) \leq E(\phi_{\mu},\mathbf{R}),
     \end{equation}
     where $E(\phi_{\mu}, \mathbf{R})$ represents the energy of the ground state $\phi_{\mu}$ on the real line.
     \label{energy_runaway}
 \end{proposition}
\begin{proof}
    Given $\mu >0$, consider a sequence of functions $u_n \in H^1(\mathbf{R})$ with compact support and mass $||u_n||^2_{L^2(\mathbf{R})} = \mu$ for any $n \in \mathbf{N}$, and such that $u_n \to \phi_{\mu}$ in $H^1(\mathbf{R})$ as $n \to \infty$. Since $u_n \to \phi_{\mu}$ in $L^p(\mathbf{R})$, we have
    \begin{equation}
        E(u_n, \mathbf{R}) \to E(\phi_{\mu}, \mathbf{R}), \qquad \text{as} \quad n \to \infty.
    \end{equation}
    Then, consider a sequence $\{x_n\} \subset \mathbf{R}$ such that $u_n(\cdot-x_n)$ is supported in $[0,+\infty)$ for any $n \in \mathbf{N}$. By identifying this interval with one half-line of $\mathcal{G}$, we can consider $u_n(\cdot-x_n)$ as a function in $H^1(\mathcal{G})$, by extending it to zero on any other edge of $\mathcal{G}$. Then we have
    \begin{equation}
        \mathcal{E}_{\mathcal{G}}(\mu) \leq \lim _{n \to \infty}E(u_n(\cdot-x_n) , \mathcal{G})= \lim _{n \to \infty}E(u_n, \mathbf{R}) = E(\phi_{\mu}, \mathbf{R}). 
    \end{equation}
\end{proof}
 
\subsection{Concentration-compactness principle}
\label{sect: concentration_comp}
The concentration-compactness principle is a variational tool that has been introduced by P.L. Lions \cite{lions} to describe the possible behaviors of sequences with prescribed $L^2$-norm in $\mathbf{R}^d$, with $d \in \mathbf{N}$. In the case of non-compact graphs, this principle has been suitably adapted in \cite{CFN17}, and can be summarized as follows.\\
Let $\mathcal{G}$ be a connected non-compact graph. For any function $\psi \in L^2(\mathcal{G})$ and $t \geq 0$ we define the concentration function
\begin{equation}
    \rho(\psi,t) = \sup_{\underline{y}\in \mathcal{G}}||\psi||^2_{L^2(B(\underline{y},t))},
\end{equation}
where $B(\underline{y},t)$ represents the open set of points on $\mathcal{G}$ whose distance from $\underline{y}$ is less than $t\geq0$. 
Given a sequence $\{\psi_n\}\subset L^2(\mathcal{G})$ such that $||\psi_n||^2_{L^2(\mathcal{G})} \to \mu >0$, the concentration parameter $\tau\in [0,\mu]$ is
\begin{equation}
    \tau = \lim_{t \to +\infty}\liminf_{n \to \infty}\rho(\psi_n,t)
\end{equation}
We can now state the following proposition (see \cite{CFN17} Section 3 and in particular Lemma 3.7).
\begin{proposition}
    Let $\{u_n\}$ be a sequence in $L^2(\mathcal{G})$ such that 
    \begin{equation}
        \int_{\mathcal{G}}|u_n|^2dx \to \mu \qquad \text{as} \quad n \to +\infty
    \end{equation}
    and
    \begin{equation}
        \sup_{n \in \mathbf{N}}||u'_n||_{L^2(\mathcal{G})} < \infty
    \end{equation}
    where $\mu >0$ is fixed. There exists a subsequence $\{u_{n_k}\}$ satisfying one of the following possibilities:
    \begin{itemize}
        \item[i)] (compactness) If $\tau = \mu$, one of the following cases occurs:
        \begin{itemize}
          \item[$i_1$)](convergence) there exists $u \in H^1(\mathcal{G})$ such that $u_{n_k} \to u$ in $L^q(\mathcal{G})$ as $k \to \infty$ for all $2\leq q\leq \infty$;
        \item[$i_2$)] (runaway) there exists one half-line $e^*$ such that for any $t>0$ and $2 \leq q\leq \infty$ we have
        \begin{equation}
            \lim_{k \to \infty} \Bigl( \sum_{e\neq e^*} ||(u_{n_k})_{e}||^q_{L^q(I_e)} + ||(u_{n_k})_{e^*}||^q_{L^q((0,t))} \Bigr) =0.
        \end{equation}  
        \end{itemize}
        \item[ii)](vanishing)  If $\tau = 0$, then $u_{n_k} \to 0$ in $L^q(\mathcal{G})$ for any $2<q\leq \infty$;
        \item[iii)](dichotomy)  If $0<\tau < \mu$, then for any $\varepsilon>0$ there exist $k_0\geq 1$ and $u^1_{k},u^2_k \in L^2(\mathcal{G})$ satisfying for $k \geq k_0$
        \begin{equation}
            ||u_{n_k} - (u^1_k+u^2_k)||_{L^2(\mathcal{G})} \leq \varepsilon ; \qquad \Bigl|\int_{\mathcal{G}}{|u_k^1|^2dx-\tau}\Bigr| \leq \varepsilon; \qquad \Bigl|\int_{\mathcal{G}}{|u_k^2|^2dx-(\mu-\tau)}\Bigr| \leq \varepsilon
        \end{equation}
        and \begin{equation}
            supp(u_k^1) \cap supp(u_k^2)  = 0.
        \end{equation}
    \end{itemize}
    \label{concentration}
\end{proposition}
In the next sections, we will apply the concentration-compactness principle to a particular class of sequences, called minimizing sequences, which are defined as follows.
\begin{definition}\label{minimizzanti}
    Given $\mu >0$, we say that $\{\psi_n\}_{n \in \mathbf{N}} \subset H^1(\mathcal{G})$ is a \textit{minimizing sequence} if and only if
    \begin{equation}
        M(\psi_n) \to \mu \quad \text{and} \quad E(\psi_n,\mathcal{G}) \to \mathcal{E}_{\mathcal{G}}(\mu), \qquad \text{as} \quad n \to \infty.
    \end{equation}
\end{definition}
Using arguments based on energy estimates (see \cite{CFN17, Claudio18, cazenave_lions}), it can be proven that for minimizing sequences the \textit{vanishing} and \textit{dichotomy} behaviors are prohibited on any connected non-compact graph. The only possible behaviors are \textit{convergence} or \textit{runaway}.  If a minimizing sequence is runaway, it escapes toward infinity along a single half-line, in a concentrated manner. On the other hand, if a minimizing sequence is convergent, then the limit function provides a ground state for the system, as explained in the following proposition.
\begin{proposition}[\cite{adami, CFN17}]
Consider a connected non-compact graph $\mathcal{G}$ and let $\mu >0$. A minimizing sequence $\{\psi_n\}_{n \in \mathbf{N}} \subset H^1(\mathcal{G})$ at (asymptotic) mass $\mu$ is weakly compact in $H^1(\mathcal{G})$. If $\psi_n \rightharpoonup \psi$ in $H^1(\mathcal{G})$, then, either, 
    \begin{itemize}
        \item[i)] $\psi_n$ is runaway and, in particular, $\psi_n \to 0$ in $L^{\infty}_{loc}(\mathcal{G})$ and $\psi =0$, or
        \item[ii)] $\psi_n$ is convergent and $\psi \in H^1_{\mu}(\mathcal{G})$, $\psi$ is a ground state and $\psi_n \to \psi$ in $H^1(\mathcal{G}) \cap L^q(\mathcal{G})$ for any $2 \leq q \leq \infty$.
    \end{itemize}
    \label{prop: minimizing_1}
\end{proposition}
\section{The topological assumption H}
\label{H}
We are interested in a family of graphs which is characterized by a topological property, called assumption H. This property has been introduced for the first time in \cite{adami_2} (see also \cite{adami_rev}) in relation to the problem of existence of ground states on connected non-compact graphs. We recall it in the following definition.
\begin{definition}
    A connected graph $\mathcal{G}$ satisfies assumption H if and only if the following holds: every $\underline{x} \in \mathcal{G}$ lies on a trail that contains two half-lines.
\end{definition}
In other words, assumption H tells us that, given any point $\underline{x}$ in $\mathcal{G}$, it is possible to reconstruct a copy of the real line $\mathbf{R}$ that passes from $\underline{x}$, by collecting consecutive edges of $\mathcal{G}$.
This property allows one to connect the minimization problem in (\ref{minimization_problem}) to the one on the real line, as explained in the following proposition.
\begin{proposition}[\cite{adami_2}]
    Let $\mathcal{G}$ be connected and satisfy assumption H. Let $\mu >0$. Then
    \begin{equation}
        \mathcal{E}_{\mathcal{G}}(\mu) = \inf_{u \in H^1_{\mu}(\mathbf{R})}E(u,\mathbf R) = E(\phi_{\mu}, \mathbf{R}).
    \end{equation}
    \label{prop: energy}
\end{proposition}
An interesting sub-family of graphs that satisfy assumption H is made of bubble tower graphs. These graphs consist in two half-lines attached to a single chain of bubbles, as depicted in Figure \ref{fig:1}. The number of bubbles is arbitrary (but greater than zero) and the size of each bubble is also arbitrary. If a bubble is made of two edges, they must have the same length.
\begin{figure}[b]
\centering
\begin{tikzpicture}
\draw[gray, thick] (-5,-2) -- (5,-2);
\draw [gray, thick](0,-2) to [out=25, in=0] (0,-1);
\draw [gray, thick](0,-2) to [out=155, in=-180] (0,-1);
\draw [gray, dashed](0,-1) to [out=25, in=0] (0,0);
\draw [gray, dashed](0,-1) to [out=155, in=-180] (0,0);
\draw [gray, thick](0,0) to [out=25, in=0] (0,1);
\draw [gray, thick](0,0) to [out=155, in=-180] (0,1);
\draw [gray, thick](0,1) to [out=25, in=0] (0,2);
\draw [gray, thick](0,1) to [out=155, in=-180] (0,2);
\draw plot [mark=*] coordinates {(0,-2)};
\draw plot [mark=*] coordinates {(0,-1)};
\draw plot [mark=*] coordinates {(0,0)};
\draw plot [mark=*] coordinates {(0,1)};
\fill (-5,-2) node[left] {$\infty$};
\fill (5,-2) node[right] {$\infty$};
\end{tikzpicture}
\caption{Bubble tower graph} \label{fig:1}
\end{figure}
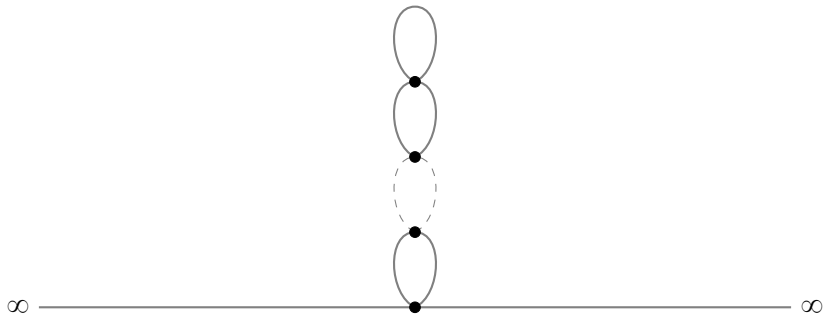
Together with the real line, bubble tower graphs are the only graphs that satisfy assumption H and admit ground states, as stated in the following proposition.
\begin{proposition}[\cite{adami_2}]
    Assume $\mathcal{G}$ is connected and satisfies assumption H. Then, for every mass $\mu >0$, $\mathcal{G}$ has no ground state unless $\mathcal{G} = \mathbf{R}$ or $\mathcal{G}$ is a bubble tower graph.
\end{proposition}
We refer to \cite{adami_2,adami_rev} for the proof and an extended explanation.
\subsection{Bubble tower graphs}\label{BBL} In this section we show that, for any mass $\mu>0$, bubble tower graphs admit a unique (up to phase shifts) ground state. \\
Suppose $\mathcal{G}$ is a bubble tower graph. Thanks to the particular symmetry of $\mathcal{G}$, we can place the positive and even ground state $\phi_{\mu}$ in (\ref{eq: soliton}) on the graph, obtaining a positive function $\Phi_{\mu} \in H^1_{\mu}(\mathcal{G})$ (we refer to the Figure \ref{fig:2} for the actual procedure, see also \cite{adami_2}). The function $\Phi_{\mu}$ has the same energy and mass of the original function $\phi_{\mu}$, i.e.
\begin{equation}
    M(\Phi_{\mu}) = \mu \quad \text{and} \quad  E(\Phi_{\mu},\mathcal{G}) = E(\phi_{\mu},\mathbf{R}),
\end{equation}
and provides the unique positive ground state of $\mathcal{G}$, as stated in the following proposition.
\begin{proposition} Consider a bubble tower graph $\mathcal{G}$ and a mass $\mu >0$. The set of ground states at mass $\mu$ is given by the orbit of $\Phi_{\mu}$ and denoted as $\mathcal{O}'_{\mu}$:
\begin{equation}
    \mathcal{O}'_{\mu} = \{ e^{i\theta}\Phi_{\mu},  \ \text{for} \  \theta \in \mathbf{R} \}.
\end{equation}
\label{ground_states}
\end{proposition}
\begin{proof}
From Proposition (\ref{prop: energy}), we have that
\begin{equation}
    \mathcal{E}_{\mathcal{G}}(\mu) = E(\phi_{\mu}, \mathbf{R}).
\end{equation}
We suppose there exists a function $\Psi \in H^1_{\mu}(\mathcal{G})$ which does not belong to $\mathcal{O}'_{\mu}$ and such that $E(\Psi,\mathcal{G}) = E(\phi_{\mu},\mathbf{R})$. We unfold the graph by considering a trail $\mathcal{T}$ that contains all the edges of $\mathcal{G}$. Following the procedure in section (\ref{sub-graphs}) we obtain a function $\Psi|_{\mathbf{R}} \in H^1_{\mu}(\mathbf{R})$ with the same mass $\mu$ and energy $E(\phi_{\mu}, \mathbf{R})$ of the soliton $\phi_{\mu}$, but such that $\Psi|_{\mathbf{R}} \notin \mathcal{O}_{\mu}$ (where $\mathcal{O}_{\mu}$ is defined in \eqref{Omu}), and this is a contradiction.
\end{proof}

\begin{figure}
\centering
	\begin{tikzpicture}
	\draw (7,0) -- (10,0);
	\draw (10,0) -- (13,0);
    \draw (10,0.75) circle  (0.75cm);
	\draw plot [mark=*] coordinates {(10,0)};
	\fill (7,0) node[left] {$\infty$};
	\fill (13,0) node[right] {$\infty$};
    \draw (10, 1.5+0.4) circle (0.4cm);
	\draw plot [mark=*] coordinates {(10,1.5)};
    \fill (10.5,2.1) node[above] {$u$};
    \fill (9,1.2) node {$v_a$};
    \fill (11,1.2) node {$v_b$};
    \fill (8,0) node[above] {$w_a$};
    \fill (12,0) node[above] {$w_b$};
	\draw [thick] plot [domain=-3:0, smooth] (\x+2, {-0.2+4/(exp(-\x)+ exp(\x))});
	\draw [thick] plot [domain=0:3, smooth] (\x+2, {-0.2+4/(exp(-\x)+ exp(\x))});
    \draw[dashed] (1.5, 0)--(1.5, 2.5);
     \draw[dashed] (2.5, 0)--(2.5, 2.5);
     \draw[dashed] (0.5, 0)--(0.5, 2.5);
     \draw[dashed] (3.5, 0)--(3.5, 2.5);
     \fill (1.8,2.3) node[right] {$u$};
     \fill (0.8,1.7) node[right] {$v_a$};
     \fill (2.8,1.7) node[right] {$v_b$};
     \fill (-0.3,0.8) node[right] {$w_a$};
     \fill (3.8,0.8) node[right] {$w_b$};
     \draw[<->] (1.5,0) -- (2.5,0);
     \fill (2,0) node[above] {\scriptsize $l_1$};
     \fill (1,0) node[above] {\scriptsize$l_2/2$};
     \fill (3,0) node[above] {\scriptsize $l_2/2$};
     \fill (2,0) node[below] {\scriptsize $I_1$};
     \fill (1,0) node[below] {\scriptsize$I_2$};
     \fill (3,0) node[below] {\scriptsize $I_3$};
     \draw[<->] (0.5,0) -- (1.5,0);
     \draw[<->] (2.5,0) -- (3.5,0);
     \draw [->] (10.5,2) to[out=-45,in=70] (10.5,1.6);
     \draw [->] (9.5,2) to[out=235,in=110] (9.5,1.6);
     \draw [->] (9.1,1) to[out=235,in=110] (9.1,0.6);
     \draw [->] (10.9,1) to[out=-45,in=70] (10.9,0.6);
     \draw [->] (8,-0.2) -- (7.5,-0.2);
     \draw [->] (12,-0.2) -- (12.5,-0.2);
	\end{tikzpicture}
    \caption{On the left we plot the profile of the  real soliton $\phi_{\mu}$ on the real line. On the right we have a bubble-tower graph $\mathcal{G}$, with two bubbles. The top bubble has perimeter $l_1$ and the bottom one has perimeter $l_2$, with $l_1<l_2$. Define the intervals $I_1,I_2,I_3 \subset \mathbf{R}$ as in the left plot and identify the portions $u, v_a,v_b, w_a,w_b$ of the soliton $\phi_{\mu}$. The function $\Phi_{\mu} \in H^1(\mathcal{G})$ is defined as follows. We place the profile of $u$ on the top bubble (with the maximum located at a distance $l_1/2$ from the vertex). The profiles of $v_a$ and $v_b$ are placed on the left and right edges of second bubble. The portions $w_a$, $w_b$ are placed on the half-lines. Gluing these profiles together we obtain the profile of $\Phi_{\mu}$. The arrows in the right picture denote the directions along which $\Phi_{\mu}$ decreases. The same procedure can be followed in the case of any number of bubbles.} 
    \label{fig:2}
	\end{figure}
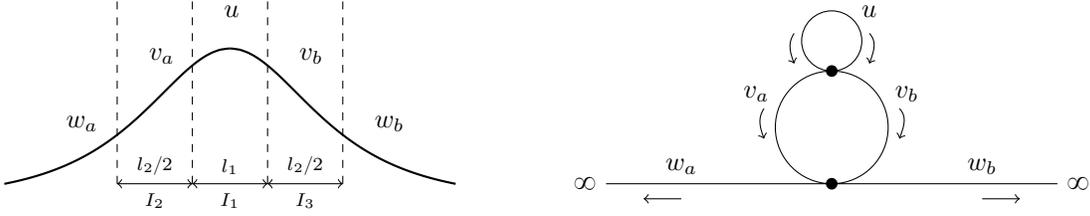
\section{Soliton confinement and reflection of slow solitons} 
\subsection{Soliton confinement and stability}
In this section we state precisely and prove the first of the main results of this paper. 
\begin{proposition}\label{main_prop_1}
    Let $\mu >0$. Assume $\mathcal{G}$ is a connected graph that satisfies assumption H, with $\mathcal{G}\neq\mathbf{R}$. Let $e$ be any of its half-lines and let $\mathcal{B}$ be the graph obtained from $\mathcal{G}$ by removing the half-line $e$. For any $\varepsilon >0$ and $M>0$ there exists $\delta >0$ and $c_{\delta}>0$ such that, if $u_0 \in H^1(\mathcal{G})$ satisfies
\begin{equation}
    ||u_0||_{H^1(\mathcal{B})} + ||(u_0)_e - \phi_{\mu}(\cdot-L)||_{H^1(\mathbf{R_+})} < \delta,
\end{equation}
with $L\geq c_{\delta}$, then the unique solution $\xi(t) \in C(\mathbf{R},H^1(\mathcal{G}))$ given by proposition (\ref{prop: cauchy}) with initial datum $u_0$ satisfies
\begin{equation}
    ||\xi(t)||_{H^1(\mathcal{B})} + \inf_{\theta \in \mathbf{R}, \  c>M}||(\xi(t))_e - e^{i\theta}\phi_{\mu}(\cdot-c)||_{H^1(\mathbf{R_+})} < \varepsilon, \qquad \forall t \geq 0.
\end{equation}
\label{main_prop}
\end{proposition}
\begin{remark}
An informal reading of Proposition~\ref{main_prop} is the following: for all times, the solution remains arbitrarily close in the energy space to a rephased and translated copy of the half-line soliton profile, with translation parameter constrained by $c>M$. Compared with the standard notion of orbital stability, the additional restriction $c>M$ provides a quantitative control of confinement: the solution stays on the same half-line and never approaches the compact core more than the prescribed distance encoded by $M$, provided that the initial datum is sufficiently close to a soliton initially placed sufficiently far away from the core.
\end{remark}

Before proving proposition (\ref{main_prop}), we introduce the following weaker statement.
\begin{lemma}\label{mainlemma}
    Let $\mu >0$ and $\mathcal{G}$ as in proposition (\ref{main_prop}). Let $e$ be any of the half-lines of $\mathcal{G}$ and let $\mathcal{B}$ be the graph obtained from $\mathcal{G}$ by removing the half-line $e$. For any $\varepsilon >0$ and $M>0$ there exists $\delta >0$ and $c_{\delta}>0$ such that, if $u_0 \in H^1(\mathcal{G})$ satisfies
\begin{equation}
    ||u_0||_{H^1(\mathcal{B})} + ||(u_0)_e - \phi_{\mu}(\cdot-L)||_{H^1(\mathbf{R_+})} < \delta,
\end{equation}
with $L\geq c_{\delta}$, then the unique solution $\xi(t) \in C(\mathbf{R},H^1(\mathcal{G}))$ given by proposition (\ref{prop: cauchy}) with initial datum $u_0$ satisfies the following property: for any $t \geq 0$ there exists a half-line $e_*=e_*(t)$ of $\mathcal{G}$ such that 
\begin{equation}
    ||\xi(t)||_{H^1(\mathcal{B_*})} + \inf_{\theta \in \mathbf{R}, \  c>M}||(\xi(t))_{e_*} - e^{i\theta}\phi_{\mu}(\cdot-c)||_{H^1(\mathbf{R_+})} < \varepsilon,
    \label{weaker}
\end{equation}
where $\mathcal{B_*}$ is the graph obtained by removing the half-line $e_*$ from $\mathcal{G}$. 
\label{weak_lemma}
\end{lemma} 
\begin{proof}
    We proceed by contradiction. We suppose there exists $\varepsilon>0$ and $M>0$, a sequence of initial data $\{u_{0,n}\} \subset H^1(\mathcal{G})$ and a sequence of positive numbers $c_n \to +\infty$ as $n \to \infty$ satisfying
\begin{equation}
    ||u_{0,n}||_{H^1(\mathcal{B})} + ||(u_{0,n})_{e} - \phi_{\mu}(\cdot-c_n)||_{H^1(\mathbf{R_+})} \to 0, \qquad \text{as} \quad n \to +\infty.
\end{equation}
for which, if for any $n \in \mathbf{N}$ the map $\xi_n(t) \subset C(\mathbf{R}, H^1(\mathcal{G}))$ is the unique solution to (\ref{NLS_eq}) with initial datum $u_{0,n}$, then there exists a sequence of times $\{t_n\} \subset \mathbf{R}_+$ such that for any half-line $e_*$
\begin{equation}
    ||\xi_n(t_n)||_{H^1(\mathcal{B_*})} + \inf_{\theta \in \mathbf{R}, \  c>M}||(\xi(t_n))_{e_*} - e^{i\theta}\phi_{\mu}(\cdot-c)||_{H^1(\mathbf{R_+})} \geq \varepsilon.
    \label{contrad_assumption}
\end{equation}
Firstly, we prove that $\{u_{0,n}\}$ is a minimizing sequence on $\mathcal{G}$. We have, by the contradiction hypothesis,
$$||u_{0,n}||^2_{L^2(\mathcal{G})} = ||u_{0,n}||^2_{L^2(\mathcal{B})} + ||(u_{0,n})_e||^2_{L^2(\mathbf{R}_+)} \to \mu, \qquad \text{as} \quad n \to \infty. $$
For the energy, we have
\begin{equation}
    ||u_{0,n}||^p_{L^p(\mathcal{B})} \leq C ||u_{0,n}||^{\frac{p}{2}+1}_{L^2(\mathcal{B})} ||u'_{0,n}||^{\frac{p}{2}-1}_{L^2(\mathcal{B})} \to 0, \qquad \text{as} \quad n \to \infty,
\end{equation}
where we used the Gagliardo-Nirenberg inequality for the connected non-compact graph $\mathcal{B}$. We conclude that
$$||u_{0,n}||^p_{L^p(\mathcal{G})} = ||u_{0,n}||^p_{L^p(\mathcal{B})}+||(u_{0,n})_e||^p_{L^p(\mathbf{R}_+)} \to ||\phi_{\mu}||^p_{L^p(\mathbf{R})}, \qquad \text{as} \quad n \to \infty.$$
Finally, we have
\begin{equation}
    ||u'_{0,n}||^2_{L^2(\mathcal{G})} = ||u'_{0,n}||^2_{L^2(\mathcal{B})} + ||(u'_{0,n})_e||^2_{L^2(\mathbf{R}_+)} \to ||\phi'_{\mu}||^2_{L^2(\mathbf{R})}, \qquad \text{as} \quad n \to \infty.
\end{equation}
We conclude that
\begin{equation}\label{important}
    E(u_{0,n}, \mathcal{G}) \to E(\phi_{\mu}, \mathbf{R}) = \mathcal{E}_{\mathcal{G}}(\mu) \qquad \text{and} \qquad M(u_{0,n}) \to \mu, \qquad \text{as} \quad n \to \infty.
\end{equation}
By conservation of mass and energy, we have
\begin{equation}
    E(\xi_n(t_n), \mathcal{G}) = E(u_{0,n}, \mathcal{G}) \to \mathcal{E}_{\mathcal{G}}(\mu) \qquad \text{and} \qquad M(\xi_n(t_n)) = M(u_{0,n}) \to \mu, \qquad \text{as} \quad n \to \infty.
\end{equation}
So $\{\xi_n(t_n)\}$ is a minimizing sequence. Using Proposition \ref{prop: minimizing_1}, we suppose $\{\xi_n(t_n)\}$ is runaway, escaping along some half-line of $\mathcal{G}$, that we call $e_*$. Suppose for the moment that $e$ is different from $e_*$. Consider a path $\mathcal{P} =(e,...,e_*)$ connecting the half-lines $e$ and $e_*$. We define $\xi_n(t_n)|_{\mathbf{R}}$ as the function in $H^1(\mathbf{R})$ obtained by restricting $\xi_n(t_n)$ to the path $\mathcal{P}$, as in section \ref{sub-graphs}. We notice that, with this ordering of $\mathcal{P}$, the half-line $e$ is associated with the portion $\mathbf{R_-}$ of the real line. We would like to show that $\{\xi_n(t_n)|_{\mathbf{R}}\}$ is a minimizing sequence on $\mathbf{R}$, i.e. it satisfies $E(\xi_n(t_n)|_{\mathbf{R}},\mathbf{R}) \to E(\phi_{\mu},\mathbf{R})$ and $||\xi_n(t_n)|_{\mathbf{R}}||_{L^2(\mathbf{R})} \to \mu$ as $n \to \infty$.\\ We call $\mathcal{B}_{\mathcal{P}}$ the (possibly disconnected) graph obtained by removing the path $\mathcal{P}$ from the graph $\mathcal{G}$. By the runaway hypothesis, we have, for any $2 \leq q \leq \infty$,
\begin{equation}
||\xi_n(t_n)||^q_{L^q(\mathcal{B}_{\mathcal{P}})} \to 0 \qquad \text{as} \quad n \to \infty.
\label{eq: Lq decay}
\end{equation}
This implies that 
\begin{equation}
    ||\xi_n(t_n)|_{\mathbf{R}}||^2_{L^2(\mathbf{R})} = M(\xi_n(t_n)) - ||\xi_n(t_n)||^2_{L^2(\mathcal{B}_{\mathcal{P}})} \to \mu \qquad \text{as} \quad n \to \infty.
\end{equation}
Regarding the kinetic energy, we would like to show that, up to subsequence, $\lim_{n \to \infty}||\xi'_n(t_n)||^2_{L^2(\mathcal{B}_{\mathcal{P}})} =0$. Suppose by contradiction that $\liminf_{n \to \infty}||\xi'_n(t_n)||^2_{L^2(\mathcal{B}_{\mathcal{P}})} = a >0$. We choose $n_0 \in \mathbf{N}$ such that $||\xi'_n(t_n)||^2_{L^2(\mathcal{B}_{\mathcal{P}})} \geq a/2$ for any $n \geq n_0$. By choosing $n_0$ possibly bigger, we have by means of (\ref{eq: Lq decay})
\begin{equation}
    E(\xi_n(t_n)|_{\mathbf{R}}, \mathbf{R}) = E(\xi_n(t_n),\mathcal{G}) - \frac{1}{2}||\xi'_n(t_n)||^2_{L^2(\mathcal{B}_{\mathcal{P}})}+ \frac{1}{p} ||\xi_n(t_n)||^p_{L^p(\mathcal{B}_{\mathcal{P}})} \leq E(\phi_{\mu}, \mathbf{R}) - \frac{a}{8} \qquad \forall n \geq n_0.
\end{equation}
By continuity of the map $\nu \to E(\phi_{\nu}, \mathbf{R})$ for any $\nu >0$, this leads to a contradiction. Indeed, the mass of $\xi_n(t_n)|_{\mathbf{R}}$ is arbitrarily close to $\mu$ for $n$ big enough, but the energy $E(\xi_n(t_n)|_{\mathbf{R}}, \mathbf{R})$ remains smaller than $E(\phi_{\mu}, \mathbf{R})$ by a factor $a/8>0$ for any $n$ big enough. We conclude that $\liminf_{n \to \infty}||\xi'_n(t_n)||^2_{L^2(\mathcal{B}_{\mathcal{P}})} = 0.$ Up to choosing a subsequence, we have 
\begin{equation}
    \lim_{n \to \infty}||\xi'_n(t_n)||^2_{L^2(\mathcal{B}_{\mathcal{P}})} = 0.
    \label{kin decay}
\end{equation}
We conclude that $E(\xi_n(t_n)|_{\mathbf{R}}, \mathbf{R}) \to E(\phi_{\mu}, \mathbf{R})$ as $n \to \infty$, and that $\{\xi_n(t_n)|_{\mathbf{R}}\} \subset H^1(\mathbf{R})$ is a minimizing sequence. By the concentration--compactness principle on the real line \cite{lions}, there exist two sequences $\{\theta_n\} \subset \mathbf{R}$ and $\{x_n\} \subset \mathbf{R}$ such that
\begin{equation}
    ||\xi_n(t_n)|_{\mathbf{R}} - e^{i\theta_n}\phi_{\mu}(\cdot - x_n)||_{H^1(\mathbf{R})} \to 0, \qquad \text{as} \quad n \to \infty.
\end{equation}
Since the sequence $\{\xi_n(t_n)\} \subset H^1(\mathcal{G})$ is runaway along $e_*$, we have $x_n \to +\infty$ as $n \to \infty$. This implies that
\begin{equation}
||\xi_n(t_n)|_{\mathbf{R}}||_{H^1(\mathbf{R}_-)} \to 0, \qquad \text{as} \quad n \to \infty,
\label{kin_R-}
\end{equation}
and that 
\begin{equation}
||\xi_n(t_n)|_{\mathbf{R}}||_{H^1_{loc}(\mathbf{R})} \to 0, \qquad \text{as} \quad n \to \infty.
\label{kin_loc}
\end{equation}
We conclude that
\begin{equation}
    ||(\xi_n(t_n))_{e_*} - e^{i\theta_n} \phi_{\mu}(\cdot-x_n)||_{H^1(\mathbf{R}_+)} \leq ||\xi_n(t_n)|_{\mathbf{R}} - e^{i\theta_n} \phi_{\mu}(\cdot-x_n)||_{H^1(\mathbf{R}_+)} \to 0, \qquad \text{as} \quad n \to \infty,
\end{equation}
and, together with (\ref{eq: Lq decay}), (\ref{kin decay}),  that
\begin{equation}
    ||\xi_n(t_n)||_{H^1(\mathcal{B}_*)} \to 0, \qquad \text{as} \quad n \to \infty.
\end{equation}
This conclusion is in contradiction with the assumption in (\ref{contrad_assumption}).\\ Suppose now that $e$ coincides with $e_*$, i.e. the sequence $\{\xi_n(t_n)\}$ escapes along the half-line $e$. In this case we choose another half-line $d$ of $\mathcal{G}$ and we consider a path $\mathcal{P}$ connecting $d$ to $e$. We can now repeat the argument above to conclude
\begin{equation}
    ||(\xi_n(t_n))_{e} - e^{i\theta_n} \phi_{\mu}(\cdot-x_n)||_{H^1(\mathbf{R}_+)} \to 0, \qquad \text{as} \quad n \to \infty,
\end{equation}
and 
\begin{equation}
    ||\xi_n(t_n)||_{H^1(\mathcal{B})} \to 0, \qquad \text{as} \quad n \to \infty.
\end{equation}
Also in this case we reach a contradiction with the assumption in (\ref{contrad_assumption}).\\
We conclude that either the assumption in (\ref{contrad_assumption}) or the runaway behavior for the minimizing sequence $\{\xi_n(t_n)\}$ is excluded. If $\mathcal{G}$ is \textit{not} a bubble tower graph, the minimizing sequence $\{\xi_n(t_n)\}$ is runaway (due to the absence of ground states). This leads us to conclude that, for $\mathcal{G}$ different from a bubble tower graph, the assumption in (\ref{contrad_assumption}) leads to contradiction and the lemma is proved.\\ Let's prove the lemma in the case $\mathcal{G}$ is a bubble tower graph and $e$ is one of its half-lines. We need to show that the minimizing sequence $\{\xi_n(t_n)\} \subset H^1(\mathcal{G})$ cannot be convergent. We proceed by contradiction. Suppose there exists $\theta \in \mathbf{R}$ such that 
\begin{equation}
    ||\xi_n(t_n)-e^{i\theta}\Phi_{\mu}||_{H^1(\mathcal{G})} \to 0 \qquad \text{as} \quad n \to \infty.
\end{equation}
where $\Phi_{\mu} \in H^1(\mathcal{G})$ is the positive ground state in (\ref{ground_states}).
Consider the functional $F:L^2(\mathcal{G}) \to \mathbf{R}$ defined as
\begin{equation}\label{F}
    F(u) = \int_{\mathcal{G}}|u|\Phi_{\mu}dx.
\end{equation}
where $\Phi_{\mu}$ is the positive ground state from Proposition \ref{ground_states}. The functional $F$ is continuous in $L^2(\mathcal{G})$, indeed if $f,g \in L^2(\mathcal{G})$
\begin{equation}
    \begin{split}
        |F(f) - F(g)| &= |\int_{\mathcal{G}}(|f|-|g|) \Phi_{\mu} dx| \leq \int_{\mathcal{G}}\bigl| |f|-|g|\bigr| \Phi_{\mu} dx \\ & \leq \int_{\mathcal{G}}\bigl|f-g\bigr| \Phi_{\mu} dx  \leq ||f-g||_{L^2(\mathcal{G})}||\Phi_{\mu}||_{L^2(\mathcal{G})}.
    \end{split}
\end{equation}
We conclude that $F(\xi_n(t_n)) \to \mu$ as $n \to +\infty$. We choose $n_0 \in \mathbf{N}$ such that $F(\xi_n(t_n)) \geq \mu /2$ for any $n \geq n_0$. Then we have $F(u_{0,n}) \to 0$ as $ n \to \infty$. Indeed, we have
\begin{equation}
    \int_{\mathcal{B}}|u_{0,n}|\Phi_{\mu}dx \to 0 \qquad \text{as} \quad n \to \infty,
\end{equation}
where the integral is computed along every edge of $\mathcal{G}$ except for the half-line $e$.
Then, for $l>0$,
\begin{equation}
    \begin{split}
        \int_{\mathbf{R_+}}(u_{0,n})_e(\Phi_{\mu})_edx &=\int_{[0,l]} (u_{0,n})_e(\Phi_{\mu})_e dx + \int_{l}^{+\infty}(u_{0,n})_e(\Phi_{\mu})_edx \\ & \leq ||(u_{0,n})_e||_{L^2([0,l])} \sqrt{\mu} + ||u_{0,n}||_{L^{\infty}(\mathcal{G})} \int_{l}^{+\infty}(\Phi_{\mu})_edx \\ & \leq \delta(n,l)\sqrt{\mu} + C \int_{l}^{+\infty}(\Phi_{\mu})_edx 
    \end{split}
\end{equation}
where $\delta(n,l) = ||(u_{0,n})_e||_{L^2([0,l])}$ and $C > ||u_{0,n}||_{L^{\infty}(\mathcal{G})}$ for any $n \in \mathbf{N}$. The second integral can be made arbitrarily small, by choosing $l$ big enough. Then, for any given $l>0$, we have $\delta(n,l) \to 0$ as $n \to \infty$, and this shows the desired limit. We choose $n_1 \geq n_0$ such that $F(u_{0,n}) \leq \mu/2$ for $n \geq n_1$.\\
By continuity of the map $t \to F(\xi_n(t))$ for $t \in \mathbf{R}$, there exists, for any $n \geq n_1$, an intermediate time $t_{n}^* \in [0,t_n]$ such that $F(\xi_n(t_n^*)) = \mu /2$. By conservation of mass and energy, the sequence $\{\xi_n(t_n^*)\} \subset H^1(\mathcal{G})$ is a minimizing sequence. But this leads to a contradiction, by means of Proposition \ref{concentration}. Indeed, if $\{\xi_n(t_n^*)\}$ is runaway, then $F(\xi_n(t_n^*)) \to 0$ as $n \to \infty$; if $\{\xi_n(t_n^*)\}$ is convergent, then $F(\xi_n(t_n^*)) \to \mu$ as $n \to \infty$. This implies that the minimizing sequence $\{\xi_n(t_n)\}$ cannot be convergent and it is runaway. Also in the case of bubble tower graphs, we conclude that the assumption in (\ref{contrad_assumption}) leads to a contradiction, and this proves the lemma.
\end{proof}
\begin{proof}[Proof of Proposition \ref{main_prop}.] In order to prove Proposition \ref{main_prop}, it is sufficient to prove that, under the hypotheses of Lemma \ref{weak_lemma}, the inequality (\ref{weaker}) holds with $e_*=e$ for every $t \geq 0$. Moreover, it is sufficient to prove Proposition \ref{main_prop} for $\varepsilon>0$ small enough.\\
Given $\mu >0$, we call $C=||\phi_{\mu}||_{H^1(\mathbf{R}_+)}$. For $0<\varepsilon < C/4$ and $M>0$, we consider $\delta >0 $ and $c_{\delta}>0$ as in Lemma \ref{weak_lemma}. In particular, we consider $\delta < C/4$. For convenience, we define the map $t \to e_*(t)$ from the interval $[0,+\infty)$ to the set of edges $E$ of $\mathcal{G}$. This map associates to each time $t$ the edge $e_*(t)$ for which the statement (\ref{weaker}) holds at time $t$, i.e. the edge along which the soliton is located. For each $t \geq 0$ there exists in fact a unique edge $e_*(t)$ for which \eqref{weaker} holds, provided $\varepsilon$ is chosen small enough. Indeed, if two different half-lines satisfied \eqref{weaker} at the same time, then the corresponding restrictions of $\xi(t)$ would both be within $\varepsilon$ of translates of $\phi_\mu$, and therefore each of them would have $H^1(\mathbf{R}_+)$-norm larger than $C-\varepsilon$. This is impossible when $\varepsilon<C/4$, because the two half-lines are disjoint and the total $H^1$-distance between the corresponding soliton profiles would be too large.\\ At $t =0$ we have $e_*(0)=e$. We define $T_+ = \sup\{ \ T >0, \ \text{s.t.} \quad e_*(t)=e \quad \forall t \in [0,T)\}$. Suppose $T_+ =0$. Then there exists a sequence of times $t_n \to 0$ such that $e_*(t_n) \neq e$ for every $n \in \mathbf{N}$. This would imply the discontinuity of the map $t \to \xi(t)$ in $H^1(\mathcal{G})$ at $t=0$. Indeed, we have $||(\xi(t_n))_e||_{H^1(\mathbf{R}_+)} < \varepsilon$ for any $n \in \mathbf{N}$ and, for a certain $L > c_{\delta}$,
\begin{equation}
   ||(\xi(0))_e||_{H^1(\mathbf{R}_+)} \geq  ||\phi_{\mu}(\cdot-L)||_{H^1(\mathbf{R}_+)} - \delta \geq C-\frac{C}{4}  > \varepsilon
\end{equation}
We conclude $T_+>0$. Suppose now $T_+ < +\infty$. Then we have $e^*(t) = e$ for every $t < T_+$. If $e_*(T_+) \neq e$, we have a contradiction. Indeed, we could consider a sequence of times $s_n <T_+$ with $s_n \to T_+$, such that $e_*(s_n) = e$ for every $n \in \mathbf{N}$, and this is again forbidden by the continuity of the map $t \to \xi(t)$ in $H^1(\mathcal{G})$ at $t = T_+$. Finally, by repeating the argument we used at $t = 0$, we can show that there exists $\delta t >0$ such that $e_*(t)=e$ for $t \in [0,T_++\delta t)$. Hence $[0,T_+]$ is not maximal and we conclude that $T_+ = +\infty$.
\end{proof}
\begin{remark}\label{esshyps}
The truly essential properties needed in the proof of Lemma \ref{mainlemma} and of Proposition \ref{main_prop_1} are the following two: 1) the graph $\mathcal{G}$ has at least two half-lines; 2) the infimum $\mathcal{E}_{\mathcal{G}}(\mu)$ of the energy at mass $\mu$ is not below the energy $E(\phi_{\mu},\mathbf{R})$ of the soliton on the line (see equation \eqref{important}). This observation in particular shows that Assumption H  is not a necessary hypothesis. Indeed, as we will discuss in Section 6, there exist graphs which violate Assumption H, but satisfy point 1) and 2) above (and so for which Proposition \ref{main_prop_1} applies). On the other hand, 2) still rules out all the cases where $\mathcal{G}$ has ground states with energy strictly below $E(\phi_{\mu},\mathbf{R})$. Despite these extensions, the class of graphs satisfying Assumption H seems a convenient and fairly general setting to work in, containing both the cases of absence and presence of a ground state. In particular, notice that the argument needed to treat  the latter case never makes use of the special structure of the tower of bubbles.
\end{remark}
\subsection{On the dynamics of slow solitons}\label{SlowS}
In this subsection we investigate, by means of Proposition \ref{main_prop}, the dynamics of slow solitons on graphs that satisfy assumption H (assuming $\mathcal{G} \neq \mathbf{R}$). More precisely, we shall consider an approximate soliton, initially supported on a half-line $e_*$ of $\mathcal{G}$, that runs toward the nearest vertex of the graph with velocity $v >0$, or that escapes from it with velocity $v<0$. In both cases, we show that when the velocity $v \in \mathbf{R}$ is small enough, the slow soliton preserves its shape in time and it remains confined on the half-line $e_*$.\\
Consider a graph $\mathcal{G}$ as in Proposition \ref{main_prop}, and call $e_*$ one of its half-lines. Let $\chi \in C^{\infty}(\mathbf{R}_+)$ be a smooth cut-off such that $\chi(x) =0$ on $(0,1)$ and $\chi(x) = 1$ on $(2,+\infty)$. Consider  $\psi_{(x_0,v)} \in H^1(\mathcal{G})$ of the form 
\begin{equation}
    (\psi_{(x_0,v)})_e(x) = \begin{cases}
        e^{-i\frac{v}{2}x}\chi(x)\phi_{\mu}(x-x_0) \qquad\text{if} \quad e=e_*\\
        0 \qquad \qquad \qquad \qquad \ \ \qquad \text{if} \quad e \neq e_*.
    \end{cases}
	\label{slow_soliton}
\end{equation}
where $x_0 \in  \mathbf{R_+}$ and $v \in \mathbf{R}$. The function $\psi_{(x_0,v)}$ represents the initial datum of an approximate soliton supported on $e_*$ and located at $x_0 >0$, running towards the nearest vertex of the graph when $v>0$ and escaping from it when $v<0$. We have the following proposition.

\begin{proposition}\label{prop:slow-solitons}
Let $\mu>0$. Assume $\mathcal{G}$ is a connected graph that satisfies assumption H, with $\mathcal{G} \neq \mathbf{R}$. Let $e_*$ be one of its half-lines and let $\mathcal{B}_*$ be the graph obtained from $\mathcal{G}$ by removing $e_*$. Let $\psi_{(x_0,v)} \in H^1(\mathcal{G})$ be an initial datum as in \eqref{slow_soliton}, for $x_0 \in \mathbf{R}_+$ and $v \in \mathbf{R}$. Let $\xi(t)$ be the unique solution to \eqref{NLS_eq} with $\xi(0) = \psi_{(x_0,v)}$. For any $\varepsilon >0$ and $M>0$, there exists a minimal distance $D>0$ and a critical velocity $v_{cr}>0$ such that the following holds true: if the parameters $(x_0,v)$ satisfy $x_0 > D$ and $|v|<v_{cr}$, then the solution $\xi(t)$ satisfies
\begin{equation}
	||\xi(t)||_{H^1(\mathcal{B}_*)} + \inf_{\theta \in \mathbf{R}, \  c>M}||(\xi(t))_{e_*} - e^{i\theta}\phi_{\mu}(\cdot-c)||_{H^1(\mathbf{R_+})} < \varepsilon, \qquad \forall t \geq 0. 
\end{equation}
\end{proposition}

\begin{remark}Proposition \ref{prop:slow-solitons} shows that, provided the approximate slow soliton is initially located sufficiently far from the nearest vertex and has sufficiently small speed, the corresponding solution $\xi(t)$ remains confined on the original half-line $e_*$ and stays close, in the energy norm, to a rephased and translated soliton profile for all times. In the case $v>0$, namely for an incoming slow soliton, we interpret this as a coarse form of reflection by the collision with the vertex of the graph.
\end{remark}

\begin{remark}
In particular, Proposition~\ref{prop:slow-solitons} rules out two alternative scenarios for sufficiently slow incoming solitons: capture by the compact core and transmission of an $O(1)$ portion of the soliton to the other half-lines. The soliton-like component remains confined on the original half-line for all positive times.
\end{remark} 

\begin{proof} Let $\varepsilon>0$ and $M>0$, and consider $\delta>0$ and $c_{\delta}>0$ as in Proposition \ref{main_prop}. We will use Proposition \ref{main_prop} after a preliminary constant gauge rotation of the initial datum.

For $(x_0,v)$ as in the statement, define
\[
\widetilde\psi_{(x_0,v)} := e^{i\frac{v}{2}x_0}\psi_{(x_0,v)}, \qquad \widetilde\xi(t):=e^{i\frac{v}{2}x_0}\xi(t).
\]
Since equation \eqref{NLS_eq} is gauge invariant, $\widetilde\xi$ is the unique solution to \eqref{NLS_eq} with initial datum $\widetilde\psi_{(x_0,v)}$. Moreover, the conclusion of the proposition is unchanged by multiplication by a constant phase, because
\[
\inf_{\theta\in\mathbf{R},\ c>M}\|\widetilde\xi_{e_*}(t)-e^{i\theta}\phi_\mu(\cdot-c)\|_{H^1(\mathbf{R}_+)}
=
\inf_{\theta\in\mathbf{R},\ c>M}\|\xi_{e_*}(t)-e^{i\theta}\phi_\mu(\cdot-c)\|_{H^1(\mathbf{R}_+)}.
\]
Therefore it is enough to prove that $\widetilde\psi_{(x_0,v)}$ satisfies the hypotheses of Proposition \ref{main_prop}.

On the distinguished half-line $e_*$ we have
\[
(\widetilde\psi_{(x_0,v)})_{e_*}(x)=e^{-i\frac{v}{2}(x-x_0)}\chi(x)\phi_\mu(x-x_0).
\]
Hence
\begin{equation}
\begin{split}
\|(\widetilde\psi_{(x_0,v)})_{e_*}-\phi_\mu(\cdot-x_0)\|_{H^1(\mathbf{R}_+)}
&\leq \|(\chi-1)\phi_\mu(\cdot-x_0)\|_{H^1(\mathbf{R}_+)} \\
&\quad + \|(e^{-i\frac{v}{2}(\cdot-x_0)}-1)\phi_\mu(\cdot-x_0)\|_{H^1(\mathbf{R}_+)} \\
&\quad + \|(e^{-i\frac{v}{2}(\cdot-x_0)}-1)(\chi-1)\phi_\mu(\cdot-x_0)\|_{H^1(\mathbf{R}_+)}.
\end{split}
\label{reflection_comp_blue}
\end{equation}
The first and the third terms in \eqref{reflection_comp_blue} are supported where $\chi\neq 1$, hence on a fixed compact interval near the vertex. Since $\phi_\mu(\cdot-x_0)$ and its derivative converge to zero uniformly on compact sets as $x_0\to+\infty$, there exists $D\geq c_\delta$ such that, for every $x_0>D$ and every $|v|\leq 1$, the sum of those two terms is smaller than $\delta/2$.

For the middle term in \eqref{reflection_comp_blue}, the translation $y=x-x_0$ gives
\[
\|(e^{-i\frac{v}{2}(\cdot-x_0)}-1)\phi_\mu(\cdot-x_0)\|_{H^1(\mathbf{R}_+)}
\leq \|(e^{-i\frac{v}{2}y}-1)\phi_\mu(y)\|_{H^1(\mathbf{R})}.
\]
Using
\[
|e^{-i\frac{v}{2}y}-1|\leq \frac{|v|}{2}|y|,
\qquad
\partial_y\bigl((e^{-i\frac{v}{2}y}-1)\phi_\mu(y)\bigr)= -i\frac{v}{2}e^{-i\frac{v}{2}y}\phi_\mu(y)+(e^{-i\frac{v}{2}y}-1)\phi_\mu'(y),
\]
together with the exponential decay of $\phi_\mu$ and $\phi_\mu'$, we obtain
\[
\|(e^{-i\frac{v}{2}y}-1)\phi_\mu(y)\|_{H^1(\mathbf{R})}\leq C_\mu |v|.
\]
Therefore there exists $v_{cr}>0$ such that, if $|v|<v_{cr}$, then the middle term in \eqref{reflection_comp_blue} is smaller than $\delta/2$.

Since $\widetilde\psi_{(x_0,v)}=0$ on $\mathcal B_*$, we also have $\|\widetilde\psi_{(x_0,v)}\|_{H^1(\mathcal B_*)}=0$. Altogether, for $x_0>D$ and $|v|<v_{cr}$ we get
\[
\|\widetilde\psi_{(x_0,v)}\|_{H^1(\mathcal B_*)}+\|(\widetilde\psi_{(x_0,v)})_{e_*}-\phi_\mu(\cdot-x_0)\|_{H^1(\mathbf{R}_+)}<\delta.
\]
Hence the hypotheses of Proposition \ref{main_prop} are satisfied by $\widetilde\psi_{(x_0,v)}$, and the conclusion follows for $\widetilde\xi$, therefore also for $\xi$.
\end{proof}

\section{Orbital stability of ground states in bubble tower graphs}
\label{section_stability}
In this section we first recall the notion of orbital stability on a metric graph and then prove the orbital stability of the ground state orbit on bubble tower graphs. The main point is that the standard Cazenave--Lions argument does not apply directly, because on bubble tower graphs one has non-compact minimizing sequences escaping along a half-line. We therefore isolate the precise compactness hypothesis needed in the classical argument and then modify the proof accordingly.
\begin{definition}
    Consider a graph $\mathcal{G}$ and a mass $\mu >0$. Let $S_{\mu}$ be the set of ground states at mass $\mu$, and assume $S_{\mu} \neq \emptyset$. We say that $S_{\mu}$ is orbitally stable if for every $\varepsilon>0$, there exists $\delta >0$ such that for any $\Phi_0 \in H^1(\mathcal{G})$, satisfying $\inf_{u \in S_{\mu}}||u-\Phi_0||_{H^1(\mathcal{G})} < \delta$, the unique solution $t\mapsto \Phi(t) \in C(\mathbf{R}, H^1(\mathcal{G}))$ of the NLS equation (\ref{NLS_eq}) with $\Phi(0) = \Phi_0$ satisfies
    \begin{equation}
        \inf_{u \in S_{\mu}}||u-\Phi(t)||_{H^1(\mathcal{G})} < \varepsilon, \qquad \forall t\geq0.
    \end{equation}
\end{definition}
\subsection{The argument of Cazenave-Lions on non-compact graphs}
For the NLS equation \eqref{NLS_eq} on a connected non-compact graph $\mathcal{G}$, the point of the Cazenave--Lions argument for orbital stability is not simply the existence of a ground state at mass $\mu$, but the compactness of every minimizing sequence at that mass. In the present setting the relevant hypothesis is the following.

\begin{itemize}
    \item[(HP)] for some $\mu>0$, every minimizing sequence with asymptotic mass $\mu$ (see Definition~\ref{minimizzanti}) is relatively compact in $H^1(\mathcal{G})$.
\end{itemize}
This is the compactness assumption behind Case~2 on p.~552 of \cite{cazenave_lions}, suitably adapted to the present setting. In particular, (HP) implies the existence of a ground state at mass $\mu$, but the converse is false in general.
Bubble tower graphs provide exactly such a situation. They do admit a ground state at every mass $\mu>0$, but they do not satisfy (HP). Indeed, by Proposition~\ref{prop: energy} one has $\mathcal{E}_{\mathcal{G}}(\mu)=E(\phi_\mu,\mathbf{R})$. As in the proof of Proposition~\ref{energy_runaway}, one can construct a sequence $\{u_n\}\subset H^1_\mu(\mathcal{G})$ supported on a single half-line and escaping to infinity, while asymptotically reconstructing the soliton $\phi_\mu$ on the line. This is a minimizing sequence which is runaway, hence not relatively compact in $H^1(\mathcal{G})$. Therefore the classical Cazenave--Lions argument cannot be applied directly to bubble tower graphs. In the next subsection we adapt the argument and prove Proposition~\ref{main_prop_2}.

\begin{remark}
\label{remark_bubble_tower}
The same argument applies to any non-compact graph $\mathcal{G}$ for which $\mathcal{E}_{\mathcal{G}}(\mu)=E(\phi_\mu,\mathbf{R})$. In that situation there always exist runaway minimizing sequences, so the compactness assumption {\rm(HP)} fails. Consequently, even if such a graph admits ground states, the classical Cazenave--Lions method does not immediately yield orbital stability. Bubble tower graphs are therefore not isolated from this point of view; see Section~\ref{conclusions} for a further example.
\end{remark} 
\subsection{Proof of orbital stability}
\begin{proposition}\label{main_prop_2}
    Let $\mathcal{G}$ be a bubble tower graph and let $\mu>0$. Then the ground state orbit
    \[
    \mathcal{O}'_{\mu}:=\{e^{i\theta}\Phi_{\mu}:\ \theta\in\mathbf{R}\},
    \]
    where $\Phi_\mu$ is given by Proposition~\ref{ground_states}, is orbitally stable.
\end{proposition}
\begin{proof}
We argue by contradiction. Assume that $\mathcal{O}'_\mu$ is not orbitally stable. Then there exist $\varepsilon>0$, a sequence of initial data $\{u_{0,n}\}\subset H^1(\mathcal{G})$, and a sequence of times $\{t_n\}\subset\mathbf{R}_+$ such that
\begin{equation}
    \|u_{0,n}-\Phi_\mu\|_{H^1(\mathcal{G})}\to0, \qquad n\to\infty,
\end{equation}
and, if $\xi_n\in C^0(\mathbf{R},H^1(\mathcal{G}))$ denotes the global solution to \eqref{NLS_eq} with initial datum $u_{0,n}$, then
\begin{equation}
    \inf_{u \in \mathcal{O}'_{\mu}}||\xi_n(t_n,\cdot) - u||_{H^1(\mathcal{G})} \geq \varepsilon \qquad \forall n \in \mathbf{N}.
    \label{contradiction}
\end{equation}
Since $u_{0,n}\to\Phi_\mu$ in $H^1(\mathcal{G})$, the Gagliardo--Nirenberg inequality \eqref{gagliardo-nirenberg} implies
\begin{equation}
    M(u_{0,n}) \to M(\Phi_{\mu}) = \mu \quad \text{and} \quad E(u_{0,n},\mathcal{G}) \to E(\Phi_{\mu},\mathcal{G}) = \mathcal{E}_{\mathcal{G}}(\mu),
    \label{limits}
\end{equation}
and conservation of mass and energy implies that
\begin{equation}
    M(\xi_n(t_n)) \to \mu \quad \text{and} \quad E(\xi_n(t_n),\mathcal{G}) \to \mathcal{E}_{\mathcal{G}}(\mu).
\end{equation}
Thus $\{\xi_n(t_n)\}$ is a minimizing sequence at mass $\mu$. By Proposition~\ref{prop: minimizing_1}, every minimizing sequence on a bubble tower graph is either relatively compact in $H^1(\mathcal{G})$ or runaway. We claim that the second alternative is impossible. Suppose by contradiction that $\{\xi_n(t_n)\}$ is runaway. Then, by Proposition~\ref{concentration},
\[
\xi_n(t_n)\to0 \qquad \text{in }L^\infty_{\mathrm{loc}}(\mathcal{G}).
\]
As in the proof of Proposition~\ref{main_prop}, we consider the continuous functional $F:L^2(\mathcal{G})\to\mathbf{R}_+$ given by
\begin{equation}
    F(u) = \int_{\mathcal{G}}|u| \Phi_{\mu} dx,
\end{equation}
Since $u_{0,n}\to\Phi_\mu$ in $H^1(\mathcal{G})$, we have $F(u_{0,n})\to F(\Phi_\mu)=\mu$. Fix $n_0$ so that $F(u_{0,n})\geq \mu/2$ for all $n\geq n_0$. On the other hand, the runaway assumption implies $F(\xi_n(t_n))\to0$ as $n\to\infty$. Indeed, this follows from the local $L^\infty$ convergence to zero together with the exponential decay of $\Phi_\mu$ along the two half-lines. More explicitly, if $e_1$ and $e_2$ denote the half-lines of $\mathcal{G}$, then for every $L>0$ we have
\begin{equation}
    \int_{L}^{+\infty}\bigl|(\xi_n(t_n))_{e_1}\bigr|\ (\Phi_{\mu})_{e_1} dx \leq ||\xi_n(t_n)||_{L^{\infty}(\mathcal{G})} \int_{L}^{+\infty}(\Phi_{\mu})_{e_1} dx,
    \label{part_of_F}
\end{equation}
and similarly for $e_2$.
Since $\{\xi_n(t_n)\}$ is uniformly bounded in $H^1(\mathcal{G})$, the tail integrals in \eqref{part_of_F} can be made arbitrarily small by choosing $L$ large, and this proves $F(\xi_n(t_n))\to0$. Choose $n_1\geq n_0$ so that $F(\xi_n(t_n))\leq \mu/2$ for all $n\geq n_1$. For each such $n$, by continuity of the map $t\mapsto F(\xi_n(t))$, there exists an intermediate time $t_n^*\in[0,t_n]$ such that
\[
F(\xi_n(t_n^*))=\mu/2.
\]
By conservation of mass and energy, the sequence $\{\xi_n(t_n^*)\}$ is again a minimizing sequence. Applying Proposition~\ref{prop: minimizing_1} once more, it is either relatively compact or runaway. If it is relatively compact, then up to a subsequence it converges in $H^1(\mathcal{G})$ to a ground state, hence $F(\xi_n(t_n^*))\to\mu$, a contradiction. If it is runaway, then the same argument as above gives $F(\xi_n(t_n^*))\to0$, again a contradiction. Therefore $\{\xi_n(t_n)\}$ cannot be runaway.
We conclude that $\{\xi_n(t_n)\}$ is relatively compact in $H^1(\mathcal{G})$. Since the ground states at mass $\mu$ form the orbit $\mathcal{O}_\mu'$, there exists $\theta\in\mathbf{R}$ such that, up to a subsequence,
\[
\xi_n(t_n)\to e^{i\theta}\Phi_\mu \qquad \text{in }H^1(\mathcal{G}),
\]
which contradicts \eqref{contradiction}. The proof is complete.
\end{proof}

\section{Extensions and some further problems}
\label{conclusions}

In this section we collect some further results and remarks on the scope of our findings and we also formulate a few open problems. Moreover, we report numerical simulations of the collision of a slow NLS soliton with the compact core of the simplest example of a metric graph satisfying Assumption~H.

\subsection{Generalizations and possible extensions: the line}
We want to extend our results to the relevant case of the NLS on the line with an inhomogeneity, not provided by the compact core of the metric graph, but by an external potential. 

Consider the focusing NLS equation in the presence of an external repulsive potential $V(x)$:
\begin{equation}
    i\partial_tu = -\partial_x^2 u - |u|^{p-2}u +V(x)u.
    \label{eq: NLS_potential}
\end{equation}
We treat two prototypical cases:\\ i) the potential $V$ is continuous non-negative and vanishing at infinity: $V(x)\geq0$ for all $x \in \mathbf{R}$, such that $V(x) \to 0$ as $|x| \to \infty$ and $V$ is different from the zero function;\\ ii) zero-range potential $V(x) = g\delta(x)$, with $g >0$, also called delta potential or, more properly, a point interaction.\\
In both cases it is well known that the Schr\"odinger operator $H=-\Delta+V$, suitably interpreted, is self-adjoint and that the subcritical NLS equation enjoys well posedness in $H^1(\mathbf{R})$.
Moreover the NLS equation conserves the mass $M(u) = ||u||^2_{L^2(\mathbf{R})}$ and the energy
\begin{equation}
    E_V(u) := \frac{1}{2}||u'||^{2}_{L^2(\mathbf{R})}-\frac{1}{p}||u||^{p}_{L^p(\mathbf{R})} + \frac{1}{2}\int_{\mathbf{R}}V(x)|u(x)|^2dx,
\end{equation}
or
\begin{equation}
     E_g(u) := \frac{1}{2}||u'||^{2}_{L^2(\mathbf{R})}-\frac{1}{p}||u||^{p}_{L^p(\mathbf{R})} + \frac{g}{2}|u(0)|^2,
\end{equation}
These facts are well known for the regular potential, while for the delta potential one can see for example \cite{GHW2004, ANV13, CFN17} and references therein.
For convenience, in this section we denote the free energy as
\begin{equation}
     E_0(u) = \frac{1}{2}||u'||^{2}_{L^2(\mathbf{R})}-\frac{1}{p}||u||^{p}_{L^p(\mathbf{R})},
\end{equation}
so that $E_V(u)= E_0(u)+\frac{1}{2}\int_{\mathbf{R}}V(x)|u(x)|^2dx$  for any $u \in H^1(\mathbf{R})$, and similarly for $E_g$.\\
Similarly to \eqref{eq: H1_mu}, we define for any $\mu >0$ the set $H^1_{\mu}(\mathbf{R}):= \{u \in H^1(\mathbf{R}), \ M(u) = \mu\}$, and we consider the following minimization problem:
\begin{equation}
    \inf_{u \in H_{\mu}^1(\mathbf{R})}E_V(u).
\end{equation}
Using the repulsive nature of the potential, and its decay at infinity, we obtain the following simple lemma, the proof of which we provide for completeness.
\begin{lemma}
    For any $\mu >0$, it holds
    \begin{equation}
        \inf_{u \in H_{\mu}^1(\mathbf{R})}E_V(u) = E_0(\phi_{\mu}),
        \label{eq: infimum_pot}
    \end{equation}
    where $ \phi_{\mu}$ represents the soliton on the real line with mass $\mu$. Moreover, the infimum is never achieved. An analogous result holds for the energy $E_g$.
\end{lemma}
\begin{proof}
   For any $u \in H^1_{\mu}(\mathbf{R})$, we have $E_V(u) \geq E_0(u) \geq E_0(\phi_{\mu})$. The first inequality follows from the fact that $V(x)\geq0$ for all $x \in \mathbf{R}$; the second from the minimality property of $\phi_{\mu}$. Next we consider, as in the proof of Proposition \ref{energy_runaway}, a sequence of functions $u_n \in H^1(\mathbf{R})$ with compact support and mass $M(u_n) = \mu$ for all $n \in \mathbf{N}$, and such that $u_n \to \phi_{\mu}$ in $H^1(\mathbf{R})$ as $n \to +\infty$. We have
    \begin{equation}
        E_0(u_n) \to E_0(\phi_{\mu}), \qquad \text{as} \quad n \to \infty.
    \end{equation}
    Then, we consider a sequence $\{y_{n}\} \subset \mathbf{R}$ such that the support of $u_n(\cdot - y_n)$ is contained in the interval $[n,+\infty)$ for all $n \in \mathbf{N}$. Then, we have 
    \begin{equation}
        \int_{\mathbf{R}}V(x)|u_{n}(x-y_n)|^2dx =  \int_{[n,+\infty)}V(x)|u_{n}(x-y_n)|^2dx \leq ||V||_{L^{\infty}[n,+\infty)}\mu \to 0
    \end{equation}
    as $n \to \infty$. We conclude that $E_V(u_n(\cdot-y_n))-E_0(\phi_{\mu}) \to 0$ as $n \to \infty,$ and that \eqref{eq: infimum_pot} holds.\\
    Next, suppose by contradiction that there exists $u \in H^1_{\mu}(\mathbf{R})$ such that $E_V(u) = E_0(\phi_{\mu})$. Then, since $E_0(u) \leq E_V(u)$, this implies $ E_0(u) \leq E_0(\phi_{\mu})$. Then it must be, up to phase shifts and translations, $u = \phi_{\mu}$. In this case we have $E_V(u) = E_V(\phi_{\mu}) > E_0(\phi_{\mu})$, which is in contradiction with the assumption $E_V(u) = E_0(\phi_{\mu})$. An analogous conclusion and proof holds if $V(x) = g\delta(x)$, with $g >0$.
\end{proof}
We can now state the following proposition, regarding the confinement of approximate solitons by means of the external potential. While the proof does not make use of essentially new ingredients with respect to the already treated case of metric graphs, we feel that it is useful to give it in detail, being the setting and scope a little different. \\
\begin{proposition}\label{Line} Let $\mu >0$. Let $V$ an external potential belonging to one of the following two classes: i) $V$ is a continuous non-negative function, such that $V$ is not identically zero and $V(x) \to 0$ as $|x| \to \infty$; ii) $V$ is a repulsive delta potential, i.e. $V(x) = g \delta(x)$, for $g >0$.\\ Then, for any $\varepsilon>0$ and $M>0$, there exists $\delta>0$ and $c_{\delta}>0$ such that, if $u_0 \in H^1(\mathbf{R})$ satisfies 
\begin{equation}
    ||u_0||_{H^1(\mathbf{R}_-)}+||u_0-\phi_{\mu}(\cdot-L)||_{H^1(\mathbf{R}_+)} < \delta,
\end{equation}
for $L \geq c_{\delta}$, then the unique solution $\xi(t) \in C(\mathbf{R}, H^1(\mathbf{R}))$ to \eqref{eq: NLS_potential} with $\xi(0) = u_0$ satisfies
\begin{equation}
    ||\xi(t)||_{H^1(\mathbf{R}_-)} + \inf_{\theta \in \mathbf{R}, \  \ c > M}||\xi(t) - e^{i\theta}\phi_{\mu}(\cdot-c)||_{H^1(\mathbf{R}_+)} < \varepsilon, \qquad \forall t \geq 0.
\end{equation}
\end{proposition}
As in the case of graphs, it seems convenient to first focus on the following weaker statement.\\
\begin{lemma}\label{Weaker statement} Let $\mu >0$. Let $V$ an external potential as above. For any $\varepsilon>0$ and $M>0$, there exists $\delta>0$ and $c_{\delta}>0$ such that, if $u_0 \in H^1(\mathbf{R})$ satisfies 
\begin{equation}
    ||u_0||_{H^1(\mathbf{R}_-)}+||u_0-\phi_{\mu}(\cdot-L)||_{H^1(\mathbf{R}_+)} < \delta,
    \label{eq: condition_u0}
\end{equation}
for $L \geq c_{\delta}$, then the unique solution $\xi(t) \in C^0(\mathbf{R}, H^1(\mathbf{R}))$ to \eqref{eq: NLS_potential} with $\xi(0) = u_0$ satisfies, at each time $t \geq 0$, one of the following two alternatives: 
\begin{equation}
    ||\xi(t)||_{H^1(\mathbf{R}_-)} + \inf_{\theta \in \mathbf{R}, \  \ c > M}||\xi(t) - e^{i\theta}\phi_{\mu}(\cdot-c)||_{H^1(\mathbf{R}_+)} < \varepsilon
    \label{eq: weaker_1}
\end{equation}
or
\begin{equation}
    ||\xi(t)||_{H^1(\mathbf{R}_+)} + \inf_{\theta \in \mathbf{R}, \  \ c < - M}||\xi(t) - e^{i\theta}\phi_{\mu}(\cdot-c)||_{H^1(\mathbf{R}_-)} < \varepsilon
    \label{eq: weaker_2}
\end{equation}
\end{lemma}
\begin{proof} We proceed by contradiction. Suppose there exist $\varepsilon>0$ and $M>0$, a sequence of initial data $\{u_{0,n}\} \subset H^1(\mathbf{R})$ and a sequence of positive numbers $c_n \to +\infty$ as $n \to \infty$ satisfying
\begin{equation}
    ||u_{0,n}||_{H^1(\mathbf{R}_-)} + ||u_{0,n} - \phi_{\mu}(\cdot-c_n)||_{H^1(\mathbf{R}_+)} \to 0, \qquad \text{as} \quad n \to \infty,
\label{eq: line_u0_contr}
\end{equation}
such that, if for any $n \in \mathbf{N}$ the map $\xi_n(t) \in C^0(\mathbf{R}, H^1(\mathbf{R}))$ denotes the unique solution to \eqref{eq: NLS_potential} with initial datum $u_{0,n}$, then there exists a sequence of times $\{t_n\} \subset \mathbf{R}_+$ such that 
\begin{equation}
    ||\xi_n(t_n)||_{H^1(\mathbf{R}_-)} + \inf_{\theta \in \mathbf{R}, \  \ c > M}||\xi_n(t_n) - e^{i\theta}\phi_{\mu}(\cdot-c)||_{H^1(\mathbf{R}_+)} \geq \varepsilon
    \label{eq: contradiction_1}
\end{equation} and
\begin{equation}
    ||\xi_n(t_n)||_{H^1(\mathbf{R}_+)} + \inf_{\theta \in \mathbf{R}, \  \ c < - M}||\xi_n(t_n) - e^{i\theta}\phi_{\mu}(\cdot-c)||_{H^1(\mathbf{R}_-)} \geq  \varepsilon
    \label{eq: contradiction_2}
\end{equation}
First, we would like to show that $\{u_{0,n}\}$ is a minimizing sequence, i.e., according to \eqref{eq: infimum_pot}, that $E_V(u_{0,n}) \to E_0(\phi_{\mu})$ and $M(u_{0,n}) \to \mu$ as $n \to \infty$. By \eqref{eq: line_u0_contr} we have, for $l>0$, 
\begin{equation} ||u_{0,n}||_{L^2(-\infty,\ l)}^2 \to 0 \qquad \text{and} \qquad ||u_{0,n}||_{L^2(l,+\infty)}^2 \to \mu \qquad \text{as} \quad n \to \infty.
\label{eq: L2_conv}
\end{equation}
We conclude $M(u_{0,n}) \to \mu$ as $n \to \infty$. Using the Gagliardo-Nirenberg inequality (see \cite{adami}) on $\mathbf{R}_-$  and \eqref{eq: line_u0_contr}, we have
\begin{equation}
        ||u_{0,n}||^p_{L^p(\mathbf{R}_-)} \leq C ||u'_{0,n}||^{\frac{p}{2}-1}_{L^2(\mathbf{R}_-)}||u_{0,n}||^{\frac{p}{2}+1}_{L^2(\mathbf{R}_-)} \to 0 \qquad \text{as} \quad n \to \infty.
\end{equation}
Similarly, we conclude that
\begin{equation}
||u_{0,n}||^p_{L^p(\mathbf{R})}  = ||u_{0,n}||^p_{L^p(\mathbf{R}_-)} + ||u_{0,n}||^p_{L^p(\mathbf{R}_+)}  \to ||\phi_{\mu}||^p_{L^p(\mathbf{R})} \qquad \text{as} \quad n \to \infty.
\end{equation}
Then, again by \eqref{eq: line_u0_contr}, we have $||u_{0,n}'||_{L^2(\mathbf{R})} \to ||\phi_{\mu}'||_{L^2(\mathbf{R})}$ as $n \to \infty$. Finally, we have for any $l >0$ and for any $n$ large enough
\begin{equation}
    \begin{split}
        \int_{\mathbf{R}}V(x)|u_{0,n}(x)|^2dx &= \int_{(-\infty,\ l)}V(x)|u_{0,n}(x)|^2dx + \int_{(l,+\infty)}V(x)|u_{0,n}(x)|^2dx \\ & \leq ||V||_{L^{\infty}(\mathbf{R})}||u_{0,n}||^2_{L^2(-\infty, \ l)} + ||V||_{L^{\infty}(l, +\infty)}2\mu.
    \end{split}
    \label{eq: V_comp}
\end{equation}
By taking $l >0$ big enough we can make the second contribution in \eqref{eq: V_comp} arbitrarily small, by the decay of $V$ at infinity. Then by taking $n$ big enough we can make also the first contribution arbitrarily small, by \eqref{eq: L2_conv}. A similar computation holds for the case $V(x) = g\delta(x)$, with $g>0$. Indeed, by the Gagliardo-Nirenberg inequality, it holds that, for any $l>0$,
\begin{equation}
    ||u_{0,n}||_{L^{\infty}(-l,l)}^2 \leq C ||u_{0,n}||_{H^1(-l,l)}||u_{0,n}||_{L^2(-l,l)} \to 0, \qquad \text{as} \quad n \to \infty.
    \end{equation}
We conclude that $E_V(u_{0,n}) \to E_0(\phi_{\mu})$ and $M(u_{0,n}) \to M(\phi_{\mu})$ as $n \to \infty$. Thus, $\{u_{0,n}\}$ is a minimizing sequence for the problem \eqref{eq: infimum_pot}.\\
By conservation of energy and mass, we have
\begin{equation}
    E_V(\xi_n(t_n)) \to E_0(\phi_{\mu}) \qquad \text{and} \qquad M(\xi_n(t_n)) \to M(\phi_{\mu}), \qquad \text{as} \quad n \to \infty.
    \label{eq: evolute_minimize}
\end{equation}
We denote $\mu_n: = M(\xi_n(t_n))$ for all $n \in \mathbf{N}$. Then, we have  $E_V(\xi_n(t_n)) \geq E_0(\xi_n(t_n)) \geq E_0(\phi_{\mu_n})$ for any $n \in \mathbf{N}$, where the second inequality follows from the minimality property of $\phi_{\mu}$, for $\mu >0$. With \eqref{eq: evolute_minimize} and by the continuity of the map $\mu \to E_0(\phi_{\mu})$ (see Equation (10) in \cite{adami}), we have $E_0(\xi_n(t_n)) \to E_0(\phi_{\mu})$ as $n \to \infty$. In other words, since it holds that
\begin{equation}
    \inf_{u \in H^1_{\mu}(\mathbf{R})}E_0(u) = E_0(\phi_{\mu}),
    \label{eq: min_line}
\end{equation}
we have that $\{\xi_n(t_n)\} \subset H^1(\mathbf{R})$ is a minimizing sequence for the problem in \eqref{eq: min_line}. The same  conclusion can be obtained in the case of a delta potential $V(x) = g \delta(x)$, with $g >0$. By the concentration--compactness principle, there exist two sequences $\{\theta_n\} \subset \mathbf{R}$ and $\{y_n\} \subset \mathbf{R}$ such that
\begin{equation}
    ||\xi_n(t_n) - e^{i\theta_n}\phi_{\mu}(\cdot-y_n)||_{H^1(\mathbf{R})} \to 0 \qquad \text{as} \quad n \to \infty.
    \label{eq: xi_conv}
\end{equation}
Our next goal is to show that $\{y_n\}$ cannot be a bounded sequence. Assume by contradiction that it is bounded. Then, up to extracting a subsequence, there exist $y_* \in \mathbf{R}$ and $\theta_* \in \mathbf{R}$ such that $y_n \to y_*$ and $\theta_n \to \theta_*$ as $n \to \infty$ (modulo $2\pi$). This implies that
\begin{equation*}
||\xi_n(t_n) - e^{i\theta_*}\phi_{\mu}(\cdot-y_*)||_{H^1(\mathbf{R})} \leq ||\xi_n(t_n) - e^{i\theta_n}\phi_{\mu}(\cdot-y_n)||_{H^1(\mathbf{R})}  + ||e^{i\theta_n}\phi_{\mu}(\cdot-y_n) - e^{i\theta_*}\phi_{\mu}(\cdot-y_*)||_{H^1(\mathbf{R})}  
\end{equation*}
$\to 0$ as $ n \to \infty$. The convergence in $H^1(\mathbf{R})$ implies that $E_V(\xi_n(t_n)) \to E_V(\phi_{\mu}(\cdot-y_*))$ as $n \to \infty$. By \eqref{eq: evolute_minimize} this implies that $E_V(\phi_{\mu}(\cdot-y_*)) = E_0(\phi_{\mu})$, which is a contradiction. We conclude that $\{y_n\}$ is unbounded. The same can be concluded in the case of a short-range potential. Suppose now $y_n \to +\infty$ as $n \to \infty$. Then, we have by \eqref{eq: xi_conv} that $||\xi_n(t_n)||_{H^1(\mathbf{R}_-)} \to 0$ and $||\xi_n(t_n)-e^{i\theta}\phi_{\mu}(\cdot-y_n)||_{H^1(\mathbf{R}_+)} \to 0$ as $n \to \infty$. This contradicts \eqref{eq: contradiction_1}. Similarly, if we suppose $y_n \to -\infty$ we find a contradiction with \eqref{eq: contradiction_2}.\\
\end{proof}
Let us now prove {\bf Proposition} \ref{Line}.
\begin{proof} We make use of Lemma \ref{Weaker statement} and then argue as in the proof of Proposition \ref{main_prop_1} for the case of metric graphs. In particular, for a given $\mu >0$, set $C := ||\phi_{\mu}||_{H^1(\mathbf{R}_+)}$. Choose $\varepsilon \in (0,C/4)$ and $M>0$. Consider $\delta >0$ and $c_{\delta}>0$ as in Lemma \ref{Weaker statement}. Choose in particular $\delta < C/4$. Let $u_0 \in H^1(\mathbf{R})$ satisfy \eqref{eq: condition_u0} and let $\xi(t)$ be the corresponding solution. By Lemma \ref{Weaker statement}, for each time $t \geq 0$ we have: (a) inequality \eqref{eq: weaker_1} holds, or (b) inequality \eqref{eq: weaker_2} holds. Our goal is to show that (a) holds for every $t \geq 0$.\\
At each $t\geq 0$ only one between (a) and (b) can hold. Indeed, if (a) holds, then $||\xi(t)||_{H^1(\mathbf{R}_-)} < \varepsilon$. While, if (b) holds, there exists $c_*>M$ such that 
\begin{equation}
    ||\xi(t)||_{H^1(\mathbf{R}_-)} \geq ||\phi_{\mu}(\cdot-c_*)||_{H^1(\mathbf{R}_-)} - 2\varepsilon \geq C - C/2 > \varepsilon.
\end{equation}
At $t =0$, we have that (a) holds. Then we define $T_+:= \sup\{t \geq 0, \ \text{s.t. \ (a) \ holds}\}$. Then, by continuity at $t=0$ of the map $t \to \xi(t)$ from $\mathbf{R}$ to $H^1(\mathbf{R})$, we have $T_+>0$. Indeed, if by contradiction $T_+=0$, there exists a sequence of times $\{t_n\} \subset \mathbf{R}_+$ with $t_n \to 0^+$ as $n \to \infty$, such that (b) holds at each $t = t_n$. This implies that $||\xi(0)||_{H^1(\mathbf{R}_-)} < \delta$, while, for a certain $L \geq c_{\delta}$,
\begin{equation}
    ||\xi(t_n)||_{H^1(\mathbf{R}_-)} \geq ||\phi_{\mu}(\cdot-L)||_{H^1(\mathbf{R}_-)} - \delta \geq C - C/4 > \varepsilon,
\end{equation}
which is a contradiction. With an analogous reasoning, we can show that $T_+ = +\infty$.
\end{proof}
Analogously to the case of metric graphs, reflection of slow solitons still holds. We state the result, avoiding the repetition of the proof, which in this case does not present any relevant difference with the one of Proposition \ref{prop:slow-solitons}. We preliminarily define the asymptotic approximated soliton:
\begin{equation}
    (\psi_{(x_0,v)})(x) = \begin{cases}
        e^{-i\frac{v}{2}x}\chi(x)\phi_{\mu}(x-x_0) \qquad\text{if} \quad x\in \mathbb{R}_+\\
        0 \qquad \qquad \qquad \qquad \ \ \qquad \text{if} \quad x\in \mathbb{R}_-   
        \end{cases}
	\label{slow_soliton2}
	\end{equation}
where $x_0 \in \mathbf{R_+}$ and $v \in \mathbf{R}$. Then, $\chi \in C^{\infty}(\mathbf{R}_+)$ is a smooth cut-off such that $\chi(x) =0$ on $(0,1)$ and $\chi(x) = 1$ on $(2,+\infty)$.	
	
 \begin{proposition}
Let $\mu>0$. Let $\psi_{(x_0,v)} \in H^1(\mathbb{R})$ be an initial datum as in \eqref{slow_soliton2}, for $x_0 \in \mathbf{R}_+$ and $v \in \mathbf{R}$. Let $\xi(t)$ the unique solution to \eqref{eq: NLS_potential} with $\xi(0) = \psi_{(x_0,v)}$. For any $\varepsilon >0$ and $M>0$, there exists a minimal distance $D>0$ and a critical velocity $v_{cr}>0$ such that the following holds true: if the parameters $(x_0,v)$ satisfy $x_0 > D$ and $|v|<v_{cr}$, then the solution $\xi(t)$ satisfies
\begin{equation}
	||\xi(t)||_{H^1(\mathbf{R}_-)} + \inf_{\theta \in \mathbf{R}, \  c>M}||\xi(t) - e^{i\theta}\phi_{\mu}(\cdot-c)||_{H^1(\mathbf{R_+})} < \varepsilon, \qquad \forall t \geq 0. 
\end{equation}
\label{proposition_slow_solitons_line}
\end{proposition}

\begin{remark}
The first analysis of cubic ($p=4$) NLS equation with a delta potential was given in \cite{GHW2004}, where an effective finite dimensional Hamiltonian system describing the dynamics of the defect-mode interaction has been introduced and numerical solutions corresponding to reflection, transmission and capture of the soliton by the defect were studied. In \cite{HMZ07} the reflection and transmission of fast solitons across a repulsive delta interaction has been initiated, exploiting the integrable character of the cubic nonlinearity, then extended to the attractive delta potential (admitting a bound state) in \cite{DatchevHolmer09}. The same problems on star graphs were studied in \cite{ACFN11}. The transmission of fast solitons across smooth potentials, possibly admitting a single bound state, is treated in the recent paper \cite{HoganMurphy24}, still for the cubic nonlinearity. The results in these papers give a control of the solution along times of the order of a certain inverse power of the velocity $v$ of the asymptotic approximate soliton. In \cite{Perelman2009} the quintic nonlinearity is considered, and the control is extended over any time.
In higher dimension we mention two relevant papers. Firstly \cite{gustafson}, where both local and nonlocal (Hartree) nonlinearities are treated with rigorous deduction of a reduced nonlinear system for the soliton parameters, still holding on finite time intervals. And finally the more recent (and more close to our main focus) \cite{NaumkinRaphael20} where the dynamics of a traveling NLS soliton escaping from a potential well in dimension $d\geq 2$ is studied, obtaining reduced equations at the leading order for the soliton center and showing a dependence of the motion of the untrapped soliton from the size of the potential tail.  

\end{remark}
\begin{remark}
As it is well known (see e.g. \cite{CGV2014} and references therein) solutions of the NLS on the line for nonlinearities $4<p<6$ that are small in suitable weighted Sobolev norms disperse and undergo scattering to free solutions. However, the states close to approximated solitons considered in our results cannot be arbitrarily small (due to the weighted norm), and there is no conflict between the two behaviors, belonging to different regions in the phase space. The same remark applies to the metric graphs case (see \cite{Aoki21}, \cite{Aoki22} for first results and examples about scattering in this setting).
\end{remark}
\subsection{Generalizations and possible extensions: non trivial metric graphs}
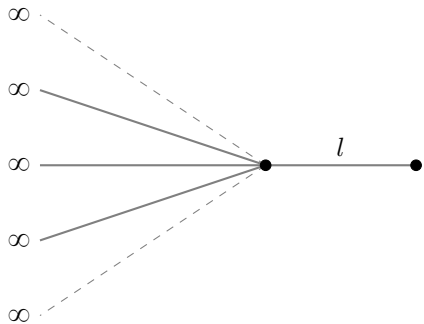
\begin{figure}[b]
\centering
\begin{tikzpicture}
\draw [gray, thick](0,0) to (-3,1);
\draw [gray, thick](0,0) to (-3,0);
\draw [gray, thick](0,0) to (-3,-1);
\draw [gray, dashed](0,0) to (-3,-2);
\draw [gray, dashed](0,0) to (-3,2);
\draw [gray, thick](0,0) to (2,0);
\draw plot [mark=*] coordinates {(0,0)};
\draw plot [mark=*] coordinates {(2,0)};
\fill (-3,1) node[left] {$\infty$};
\fill (-3,2) node[left] {$\infty$};
\fill (-3,0) node[left] {$\infty$};
\fill (-3,-1) node[left] {$\infty$};
\fill (-3,-2) node[left] {$\infty$};
\fill (1,0) node[above] {$l$};
\end{tikzpicture}
\caption{N half-lines and a pendant all emanating from the same vertex. The number of half-lines is $N \geq 3$ and the length of the pendant is $l >0$.} \label{fig:pendant}
\end{figure}
We now come back to the case of non trivial metric graphs, discussing the generality of the results here presented. First, we note that Proposition~\ref{main_prop} holds, in fact, for a wider class of graphs than those characterized by Assumption~H. Indeed, as noticed in Remark \ref{esshyps}), the proof of Proposition~\ref{main_prop} it is sufficient to assume that the following two conditions hold: 1) the connected graph $\mathcal{G}$ has at least two half-lines;  2) the infimum $\mathcal{E}_{\mathcal{G}}(\mu)$ of the energy at mass $\mu$ is not below the energy $E(\phi_{\mu},\mathbf{R})$ of the soliton on the line (see also Remark \ref{esshyps}). Of course one could consider metric graphs with non-negative delta vertices or non-negative potentials, as we did with the line, obtaining not surprising generalizations of those results, under suitable hypotheses. More interestingly, we show that there exist examples of metric graphs with Kirchhoff boundary conditions and no external potentials that satisfy conditions 1) and 2), but violate Assumption H. As we will see, when 1) and 2) hold, ground states may or may not exist.
This indicates that, while Assumption~H remains a convenient framework, the confinement or reflection behavior is not tied to topology and it is better understood as a threshold-energy mechanism.

Consider a graph $\mathcal{G}$ consisting of $N$ half-lines, with $N \ge 3$, and a pendant edge of length $l>0$, all attached to a single vertex (see Fig.~\ref{fig:pendant}). Because of the pendant edge, Assumption~H is violated. By \cite[Theorem 4.4]{adami} there exists a critical mass $\mu^*>0$ such that $\mathcal{G}$ admits ground states at mass $\mu$ if and only if $\mu\ge \mu^*$. Therefore, for any $\mu<\mu^*$ these graphs have no ground states. Thus, $\mathcal{E}_{\mathcal{G}}(\mu) = E(\phi_{\mu},\mathbf{R})$ for all $\mu <\mu^*$ and Proposition~\ref{main_prop} applies.

Next, if $\mu\ge \mu^*$, then ground states exist. Since the map $\mu\mapsto \mathcal{E}_{\mathcal{G}}(\mu)$ is continuous (see \cite{adami}, Theorem~3.1), we have
\[
\mathcal{E}_{\mathcal{G}}(\mu^*)
= \lim_{\mu\to (\mu^*)^-}\mathcal{E}_{\mathcal{G}}(\mu)
= \lim_{\mu\to (\mu^*)^-}E(\phi_{\mu},\mathbf{R})
= E(\phi_{\mu^*},\mathbf{R}).
\]
We conclude that $\mathcal{G}$ admits ground states at mass $\mu^*$ with
$\mathcal{E}_{\mathcal{G}}(\mu^*)=E(\phi_{\mu^*},\mathbf{R})$.
This example shows concretely that the confinement argument extends beyond Assumption~H, even in the presence of ground states.

We conclude with some open problems. In Section~\ref{section_stability} we discussed how the Cazenave--Lions argument for orbital stability of ground states on non-compact graphs relies on hypothesis (HP), namely the relative compactness of all minimizing sequences (with prescribed asymptotic mass). We also showed that bubble-tower graphs violate (HP) for every mass $\mu>0$. In fact, any graph $\mathcal{G}$ for which $\mathcal{E}_{\mathcal{G}}(\mu)=E(\phi_{\mu},\mathbf{R})$ holds violates (HP) at mass $\mu$, as discussed in Remark~\ref{remark_bubble_tower}. Therefore, if any of these graphs admit ground states at that mass, their orbital stability cannot be proved by the Cazenave--Lions argument. Notice that graphs of this kind exist even outside the class H, as the example in the previous paragraph shows. For such graphs, it is natural to ask whether orbital stability of the set of ground states can be proved by the same strategy as in Proposition~\ref{main_prop_2}. We believe the answer is affirmative if the set of ground states at mass $\mu^*$ consists of finitely many orbits (see \cite{DST20} for a discussion of uniqueness of ground states on this type of graphs). In this respect, we mention \cite{GG17}, where the authors study stability of ground states for the NLS equation with general nonlinearities on the line. When the set of ground states contains multiple orbits, they prove orbital stability of each orbit by a method similar to that of Proposition~\ref{main_prop_2}. We expect that their approach can be adapted to non-compact graphs with multiple ground states, but we leave this conjecture for future work.

Finally, a natural extension and refinement of Proposition~\ref{main_prop} is to show that the parameters $c$ and $\theta$ are regular (e.g.\ continuous, or even differentiable) functions of time, in the spirit of deriving modulation equations (see, e.g., \cite{bona}, or some of the paper cited in Remark 6.2, such as \cite{gustafson},\cite{NaumkinRaphael20} and reference therein). If this is possible, several interesting questions arise. For instance, one may ask whether the approximate solitons in Proposition~\ref{main_prop} drift farther and farther away from the compact core as time evolves, i.e.\ whether $c(t)\to+\infty$ as $t\to+\infty$ at some rate, or whether can stay bounded or oscillate. In this direction, we recall the work of Cavalcante and Mu\~noz \cite{CavMun23} on asymptotic stability of KdV solitons on the half-line. We partially address these questions with the aid of numerical simulations in the next section and we plan to discuss more thoroughly the problem from a rigorous perspective in a subsequent work. 

In the same spirit, one may ask whether there exist non-compact graphs $\mathcal{G}\neq\mathbf{R}$ for which slow incoming solitons exhibit a genuinely different qualitative behavior, for instance capture, transmission or escape of moving solitons. In the integrable case $p=4$, related reflection phenomena on the half-line with several boundary conditions are discussed in \cite{biondini}, and for NLS in $\mathbf R^n$ see the already cited \cite{NaumkinRaphael20}.

\subsection{Numerical simulations}
Using the QGLAB package developed by R. Goodman, G. Conte and J. Marzuola \cite{goodman} we have performed some numerical simulations in order to illustrate and to test quantitatively the results in Section \ref{SlowS}. In particular, we simulate the dynamics of a slow soliton colliding against the vertex of a star graph with three half-lines.\\
In the simulation the star graph is approximated by three segments of length $L=50$, all emanating from a single vertex. We choose the power $p = 5$ for the nonlinearity in \eqref{NLS_eq}, and we check that for other nonlinearities similar results are obtained (in particular for the integrable case $p=4$). To build the initial datum, we consider the following moving soliton on the line  \cite{sulem}
\begin{equation}
    \phi(x) = e^{-ixv_0}\Bigl(\frac{p}{2}\Bigr)^{\frac{1}{p-2}} \frac{1}{\cosh^{\frac{2}{p-2}}(\frac{p-2}{2}x)},  \qquad \text{with} \quad p = 5.
    \label{num_soliton}
\end{equation}
The simulation is started with $u_0(x)=\phi(x-L/2)$ on one edge of the graph, and $u_0(x)=\phi(x+L/2)$ on the other edges. The incoming velocity is chosen to be $v_0=-0.08$; the total time of simulation is set to $t=300$. In Figure \ref{sub_graphs} we report some snapshots of the modulus of the solution $|u(t)|^2$ at different times. The collision happens approximately at $t^*=144$; after that the soliton is completely reflected. Figure \ref{collision} displays the evolution of the slow soliton on single half-lines. We observe how the vertex repels the soliton, allowing only a minimal overlap.  In Figure \ref{mom-energy} (A) we report the ratio $p(t)/p(0)$ over time, where
\begin{equation}
    p(t) = \Im \int_{\mathcal{G}} u'\overline{u} dx
\end{equation}
is the linear momentum on $\mathcal{G}$. We observe that, before and after the collision, the momentum is approximately conserved in modulus, with a transient time around $t = 144$. On the other hand, its sign flips once the collision has happened. In Figure \ref{mom-energy} (B) we report the ratio $K(t)/K(0)$ over time, where
\begin{equation}
    K(t) = \frac{1}{2}||u'||^2_{L^2(\mathcal{G})}
\end{equation} is the kinetic energy. This quantity also has a transient time around $t=144$, but the value before and after the collision is approximately conserved, as in an elastic collision. It interesting to notice that the kinetic energy increases as the collision happens, and it reaches its maximal value at the collision time. Thus, the reflection of the soliton takes place without a classical turning point. This behavior is reminiscent of a quantum reflection \cite{Cornish08, Brand06}, which consists in the non classical reflection of slow solitons by an attractive and rapidly varying potential. We remark that, in the present case, the reflection takes place in the absence of an external potential and it is due to the vertex interaction, but the qualitative behavior we observe is analogous..\\
 Finally we checked that the total energy (\ref{energy}) is conserved with accuracy $\bigl|\frac{E(t)-E(0)}{E(0)}\bigr| < 8 \cdot 10^{-3}$ and the total mass (\ref{mass}) is conserved with accuracy $\bigl|\frac{M(t)-M(0)}{M(0)}\bigr| < 1 \cdot 10^{-3}$, over the whole simulation.\\
The results of the numerical simulations seem to indicate that the compact part of the graph has a net repulsive effect on the slow soliton (we mention in this respect the analysis of instability of half-solitons on star graphs given in \cite{KP18}). In the following, we want to investigate the spatial range of this effect. In order to do this, we consider an approximate soliton located on one half-line of the star graph, at distance $x_0>0$ from the vertex, and with zero incoming velocity. Our goal is to study how the dynamics associated to $u_0$ changes as we increase $x_0>0$. \\
We start the simulation with the following initial datum: $u_0=\phi(x-x_0)$ on one edge of the star graph and $u_0(x) = \phi(x+x_0)$ on the other two edges, where $\phi$ is given in \eqref{num_soliton}, with $v=0$. The results of the simulations are reported in Figure \ref{sub_graphs_2}.  We observe a clear repulsive effect for the case $x_0=L/10$, which pushes the soliton away from the vertex. This effect is present also for $x_0=L/9$ and $x_0=L/8$, but with milder intensity. Thus, the repulsive effect seems to diminish as we move away from the vertex. In the case $x_0=L/2$, the repulsive effect appears to be so low that the solution remains in its initial position for the whole time interval of the simulation (this static behavior is observed even if we run the simulation until $t=600$). If the value of $x_0>0$ is increased further so as to approach the free end of the edge (where Dirichlet conditions are applied), boundary effects may start to appear. We believe that these effects can be observed in Figure \ref{sub_graphs_3}, where the evolution associated to $x_0 = 0.875L$ and $x_0=0.9L$ is reported. We observe that the presence of the Dirichlet boundary has a repulsive effect as well, which pushes the soliton towards the center of the edge, and whose intensity increases as we move closer to the free end of the edge. The presence of boundary effects is indicative of the fact that the infinite length approximation, which replaces the half-lines of the graph with finite length edges, is effective provided that the amplitude $|u(t)|^2$ of the solution always remains small in the extremal regions of the compact graph. Otherwise, finite-size effects may appear. In the numerical simulations reported so far, it seems reasonable to assume that, as long as the distance between the soliton and the vertex remains smaller than half of the lentgh of the edge, boundary effect are negligible. Finally, we report in Figure \ref{sub_graphs_4} the evolution associated with $x_0=0.875L$ as before, but this time with a nonzero incoming velocity $v=-0.1$. In this case we believe boundary effect are still present, but less apparent, probably due to the fact that the solution spends a relatively short time in the vicinity of the boundary. In particular, we observe that the soliton proceeds in uniform motion towards the vertex until the collision time, at which it is reflected, in agreement with the results in Section \ref{SlowS}.


\section*{Acknowledgments.}
\noindent The authors are grateful to Roy Goodman for various comments and discussions. The second author acknowledges the support of the Next Generation EU - Prin 2022 project "Singular Interactions and Effective Models in Mathematical Physics- 2022CHELC7" and of Gruppo Nazionale di Fisica Matematica (GNFM-INdAM).


\begin{figure}
\centering
\subfloat[][\emph{$t={0}$}]
{\includegraphics[width=.40\textwidth]{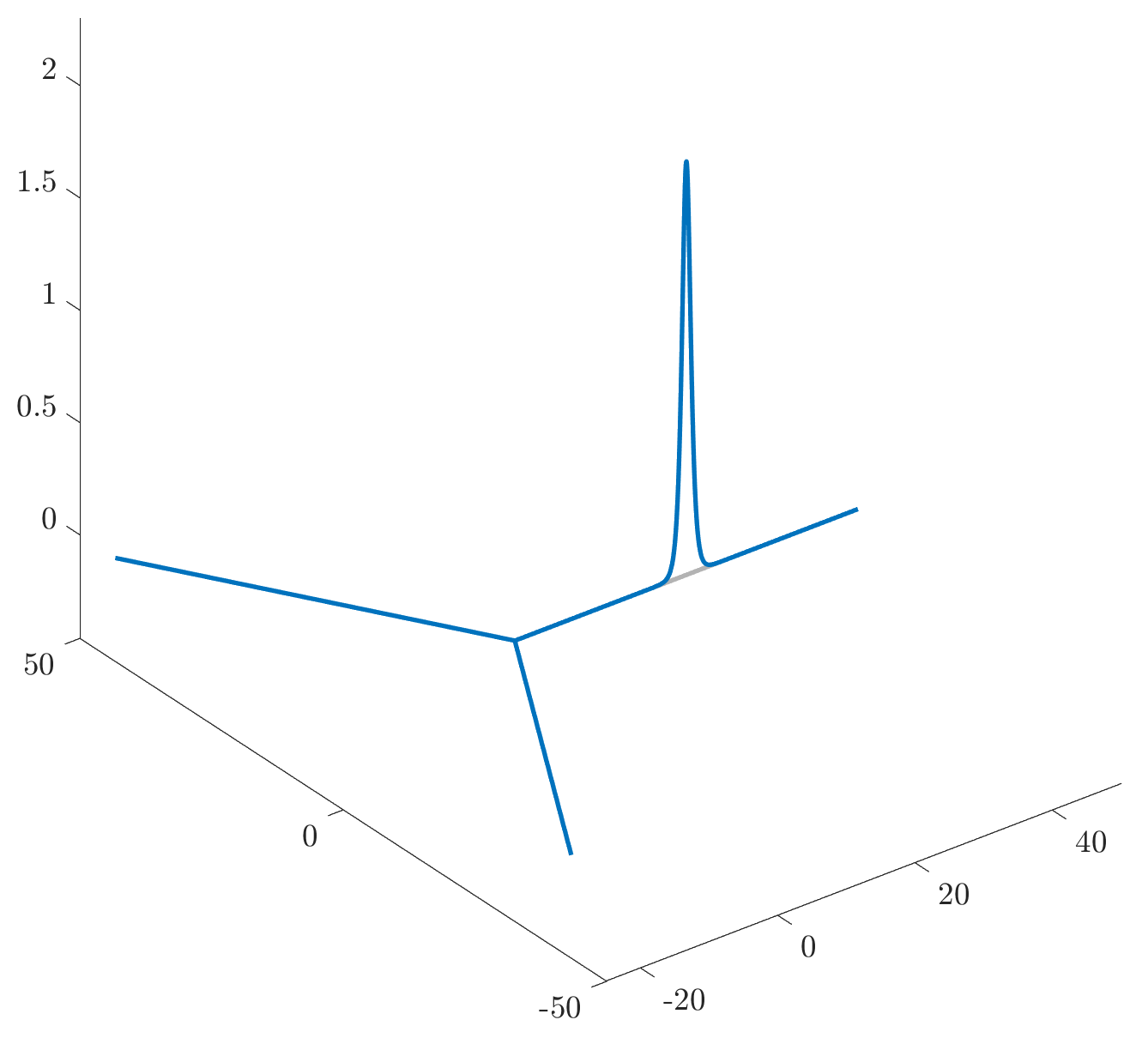}} \quad
\subfloat[][\emph{$t=100$}]
{\includegraphics[width=.40\textwidth]{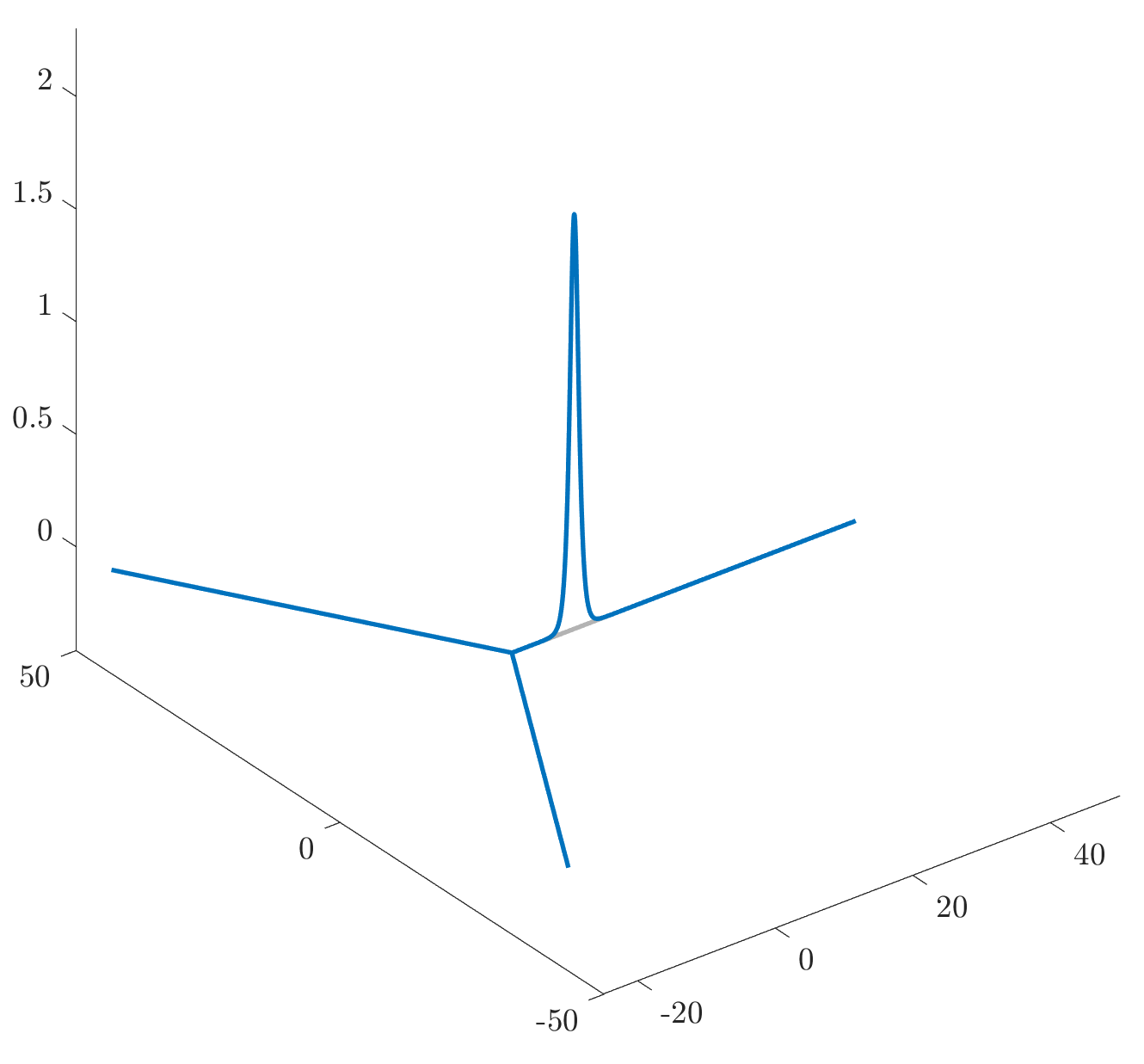}} \quad
\subfloat[][\emph{$t={144}$}]
{\includegraphics[width=.40\textwidth]{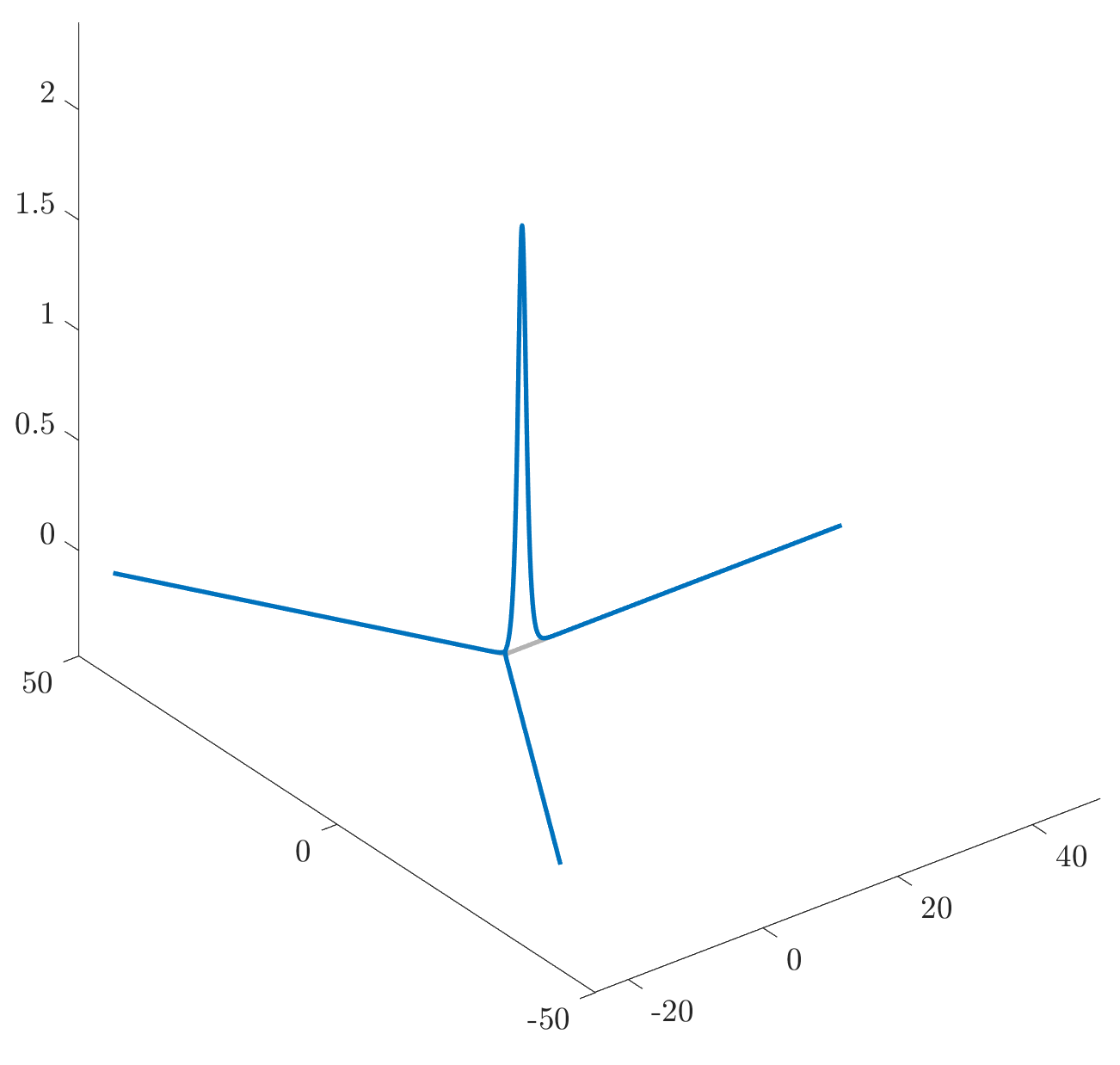}} \quad 
\subfloat[][\emph{$t={175}$}]
{\includegraphics[width=.40\textwidth]{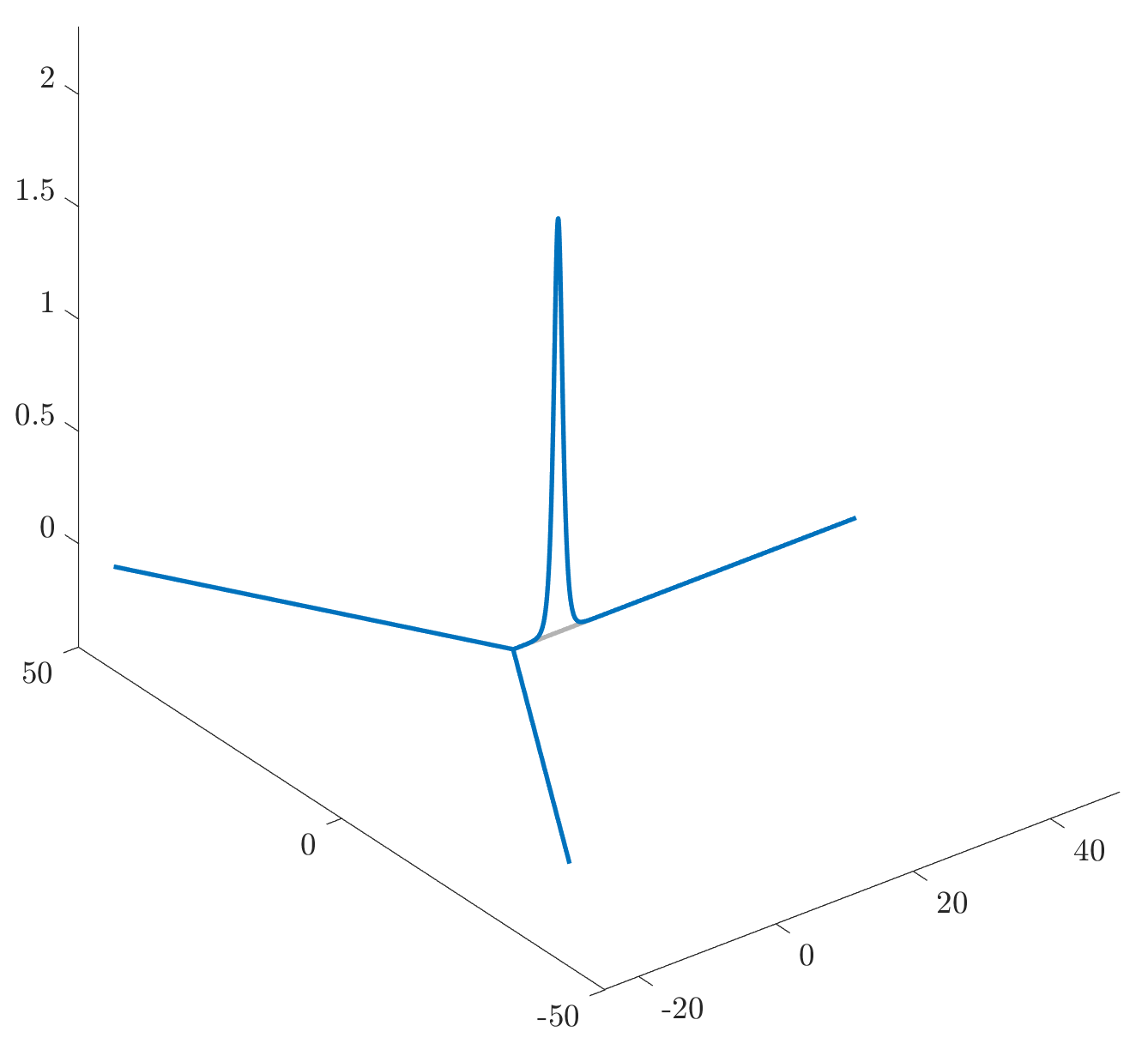}} \quad 
\subfloat[][\emph{$t={225}$}]
{\includegraphics[width=.40\textwidth]{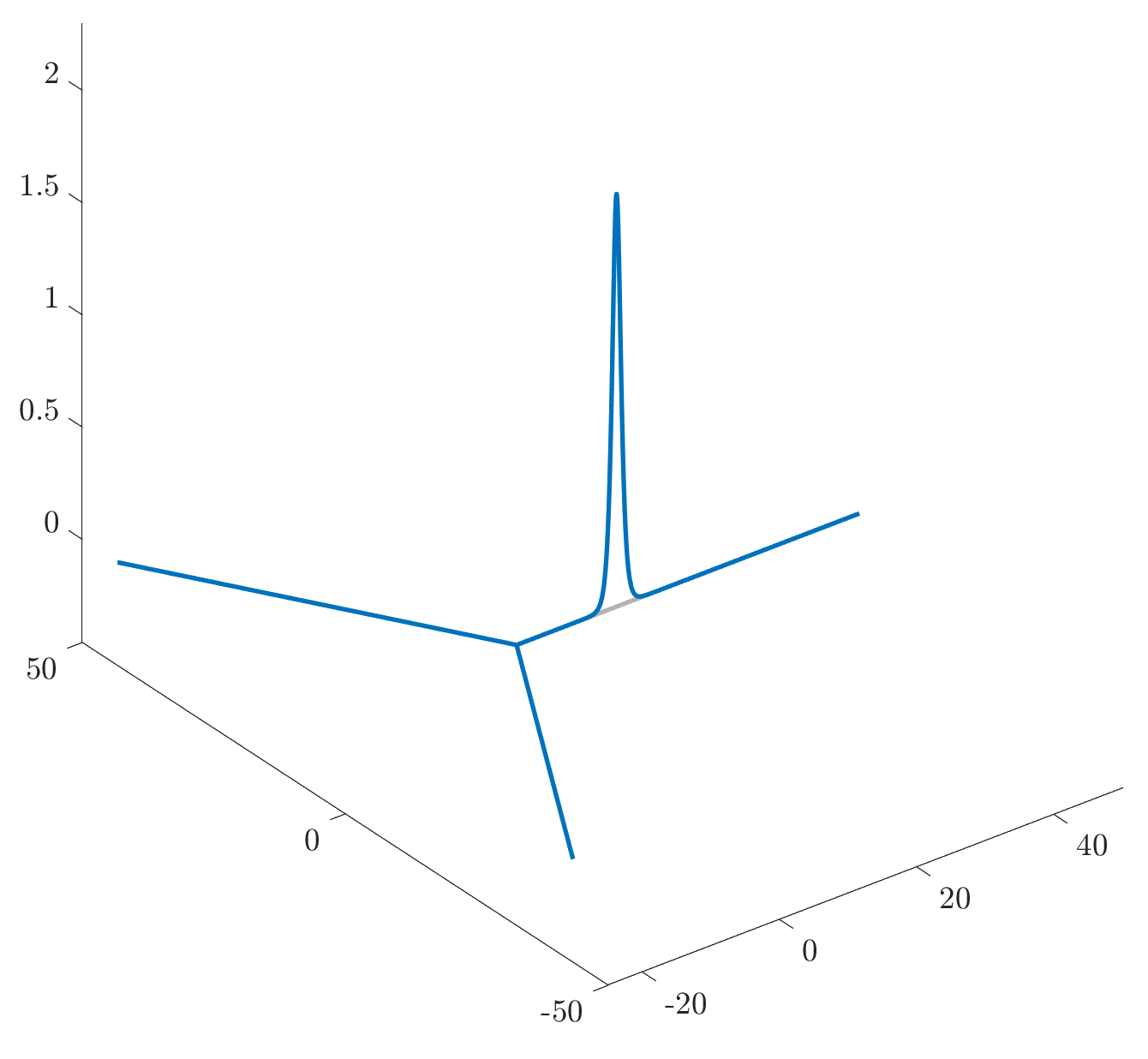}} \quad
\subfloat[][\emph{$t={300}$}]
{\includegraphics[width=.40\textwidth]{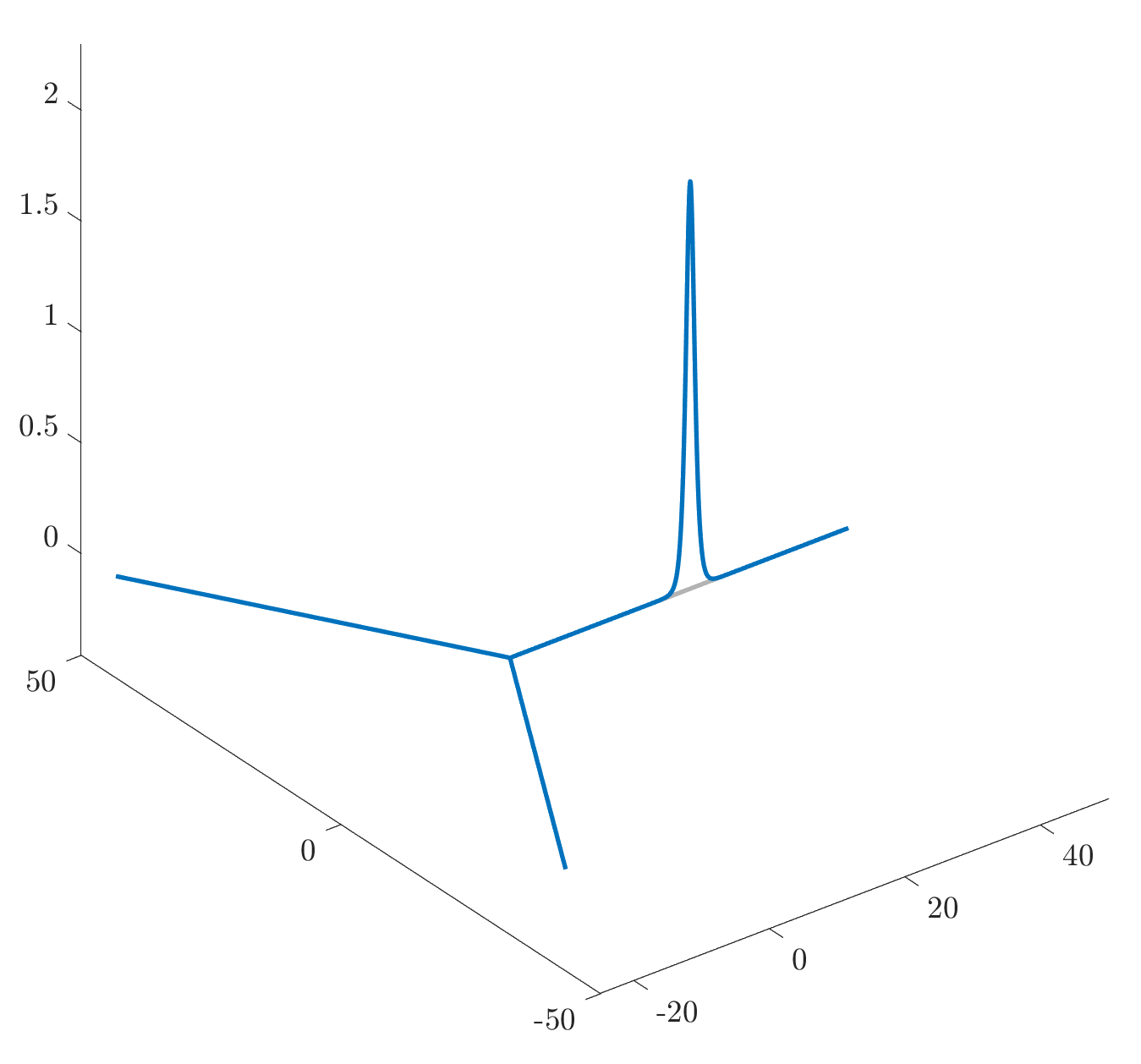}} \\
\caption{Snapshots of the amplitude $|u(t)|^2$ of the solution to the NLS equation \eqref{NLS_eq} with $p=5$ on the star graph at six different times.  The initial datum corresponds to an approximated soliton placed halfway along one edge of the graph, and with velocity $v=-0.08$ towards the vertex.  We observe the soliton approaching the vertex of the graph and hitting it at approximately $t=144$. It is then reflected completely. }
\label{sub_graphs}
\end{figure}
\begin{figure}
\centering
\subfloat[][]
{\includegraphics[width=\textwidth]{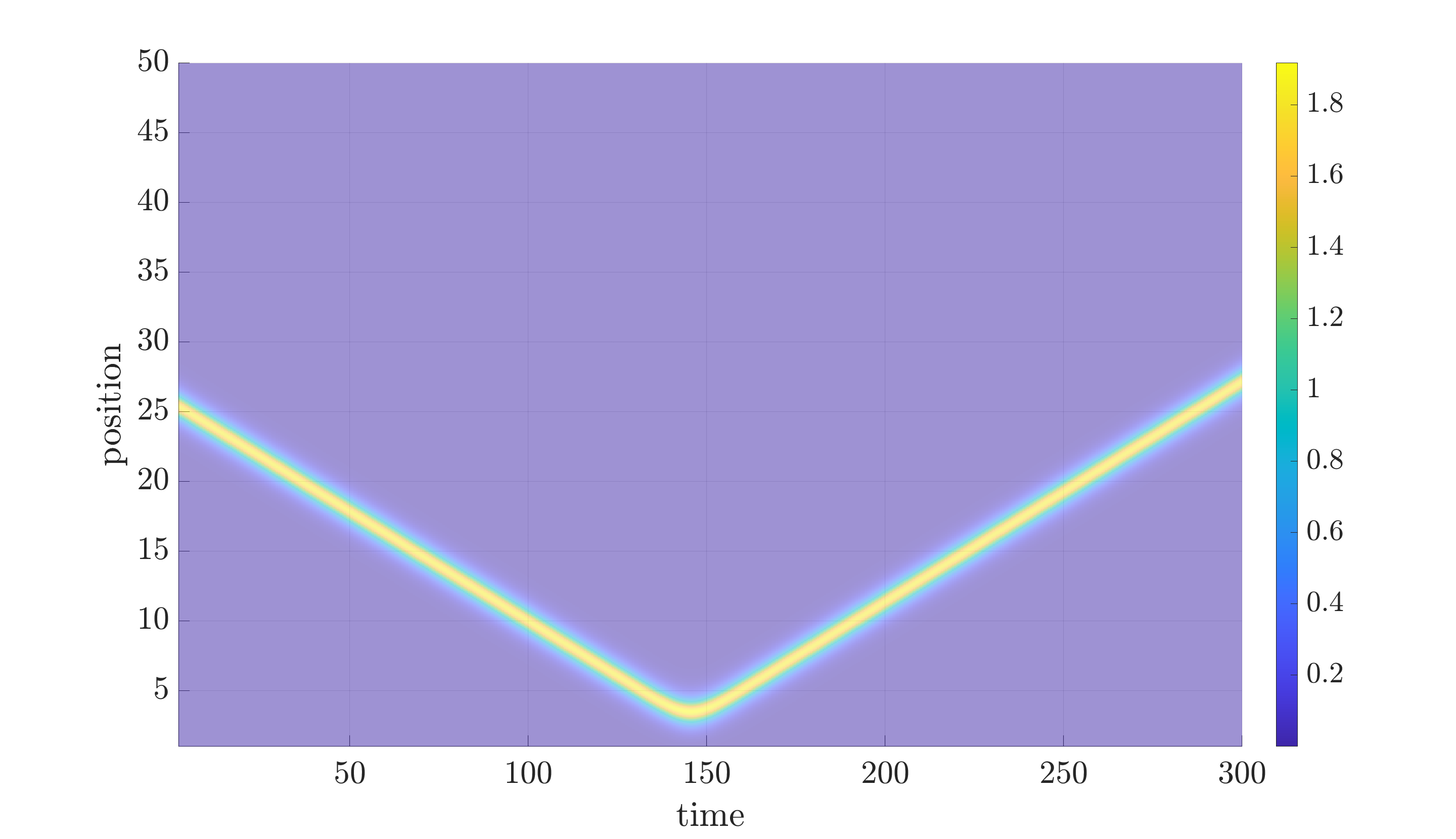}} \quad 
\subfloat[][]
{\includegraphics[width=\textwidth]{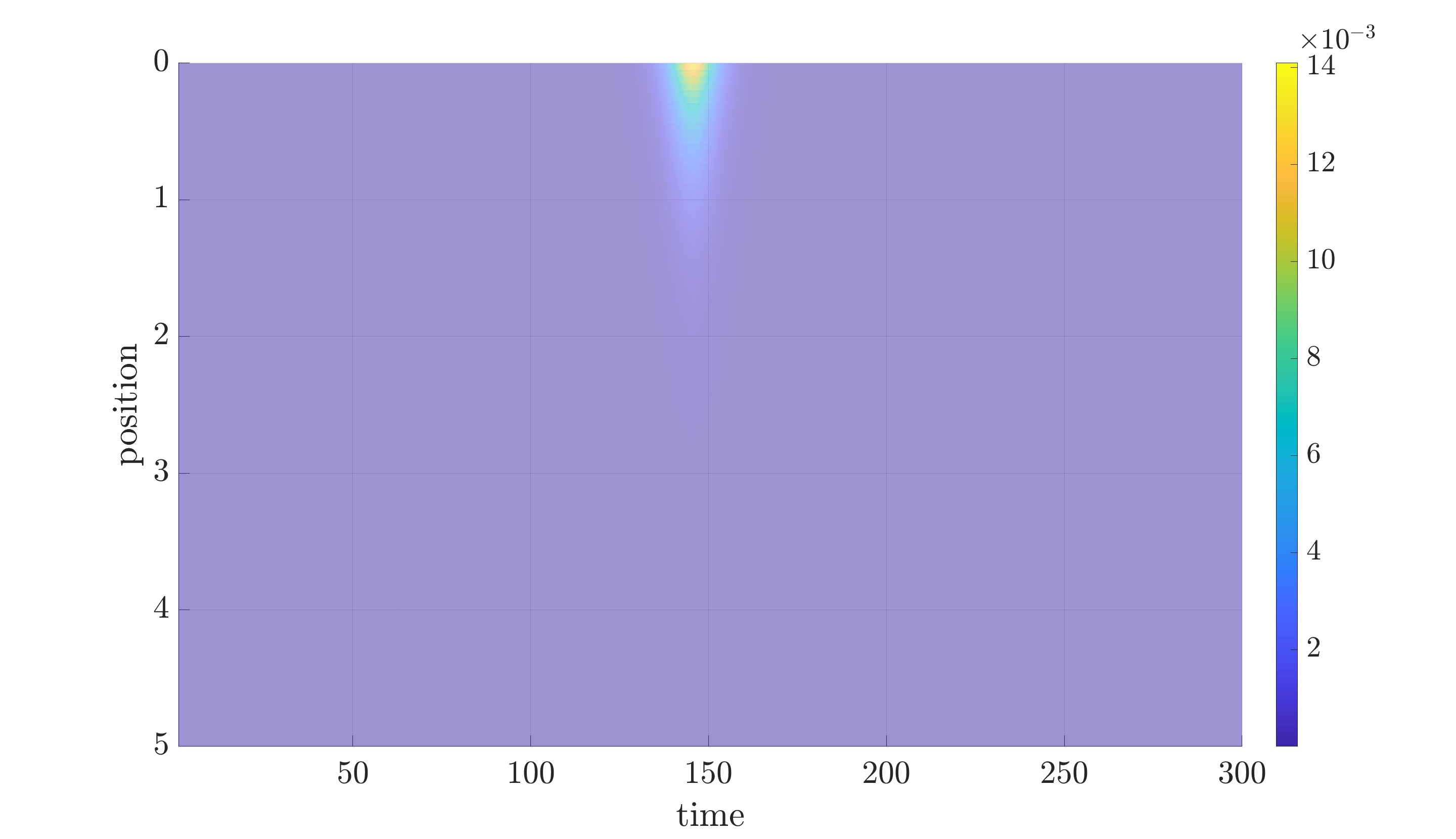}} \\
\caption{We plot the amplitude $|u(x,t)|^2$ as a function of space and time, for the same solution as in Figure \ref{sub_graphs}. In the plots, the position of the vertex of the star graph corresponds to the origin of the vertical axis. In Figure (A) we display the edge that initially contains the soliton. We observe how the soliton proceeds with uniform motion until the collision time, at which it is reflected. It then proceeds backwards, still with uniform motion. In Figure (B) we display one of the edges of the graph which are initially empty, i.e.  which initially contain only the tail of the soliton (for convenience, we display only one tenth of the total length of the edge). We observe how the soliton is unable to access this edge, and how its maximal amplitude is approximately $1.4\times10^{-2}$ at the collision time $t^* = 144$. A similar result holds for the other empty edge. }
\label{collision}
\end{figure} 
\begin{figure}
\centering
\subfloat[][]
{\includegraphics[width=\textwidth]{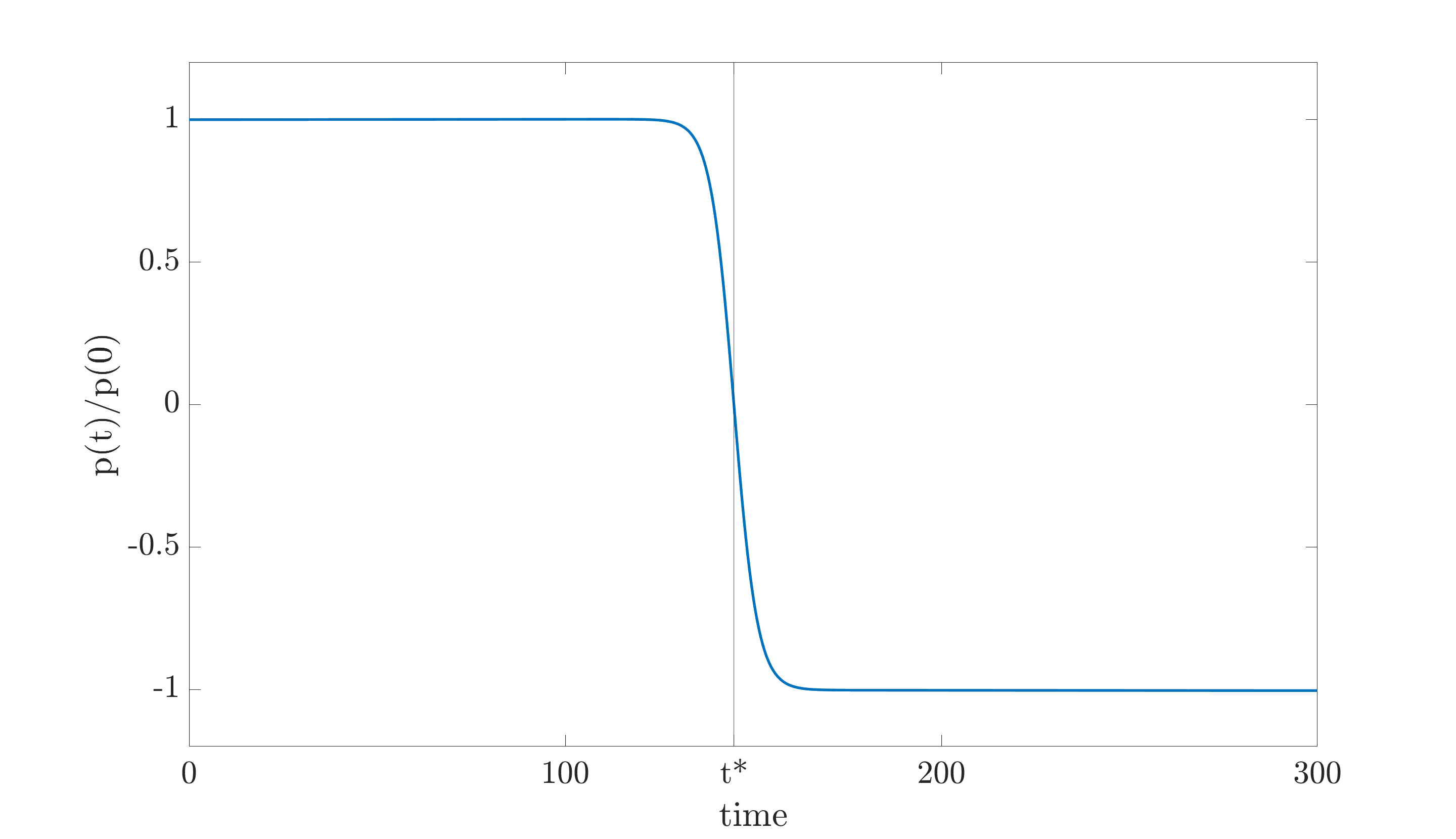}} \quad 
\subfloat[][]
{\includegraphics[width=\textwidth]{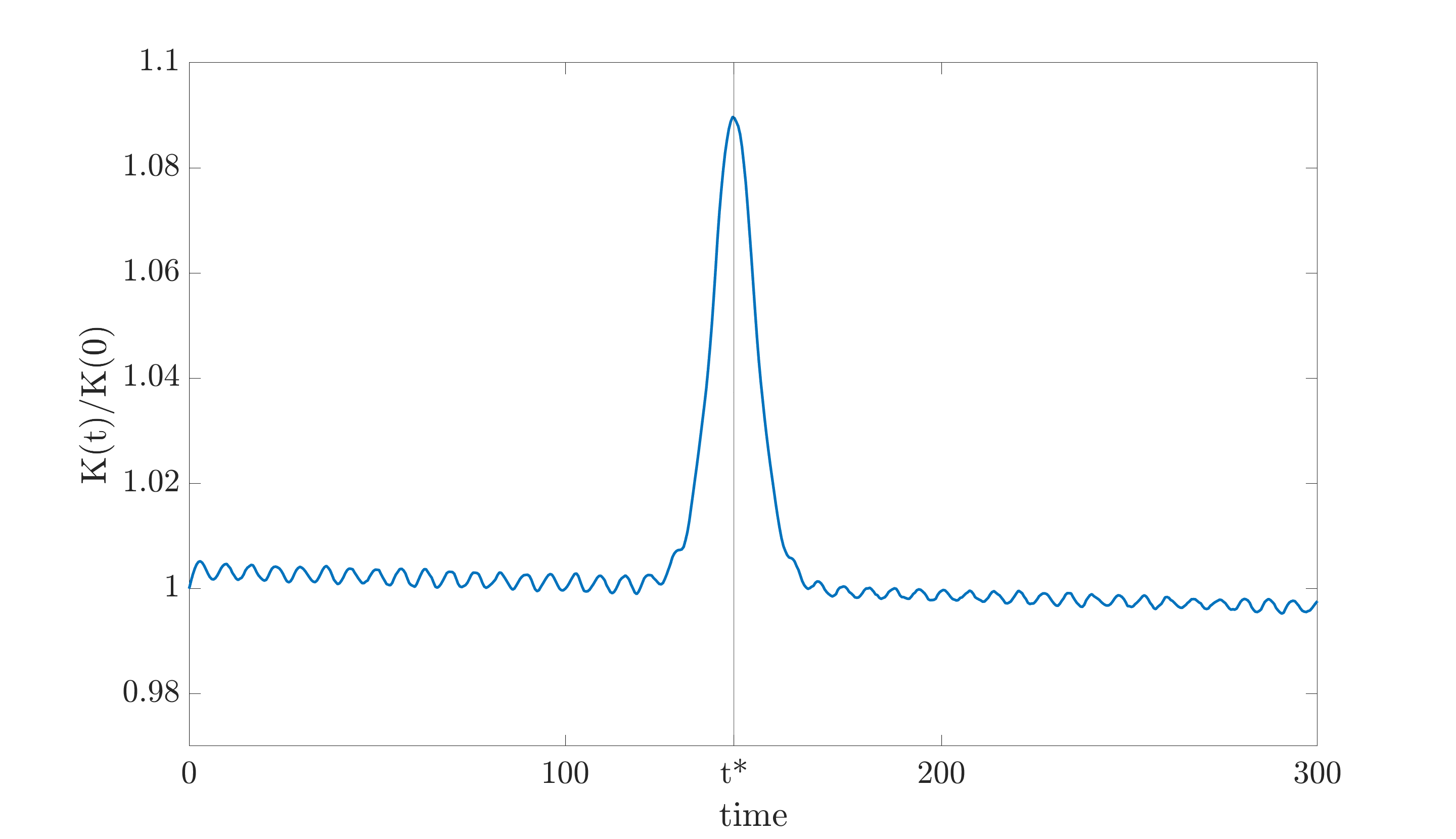}} \\
\caption{For the same solution as in Figure \ref{sub_graphs} we plot the momentum $p(t)$ divided by $p(0)$ in picture (A). In picture (B) we plot the kinetic energy $K(t)$ divided by $K(0)$. The collision time corresponds to $t^* = 144$.}
\label{mom-energy}
\end{figure} 
\begin{figure}
\centering
\subfloat[][\emph{$x_0=L/10$}]
{\includegraphics[width=.45\textwidth]{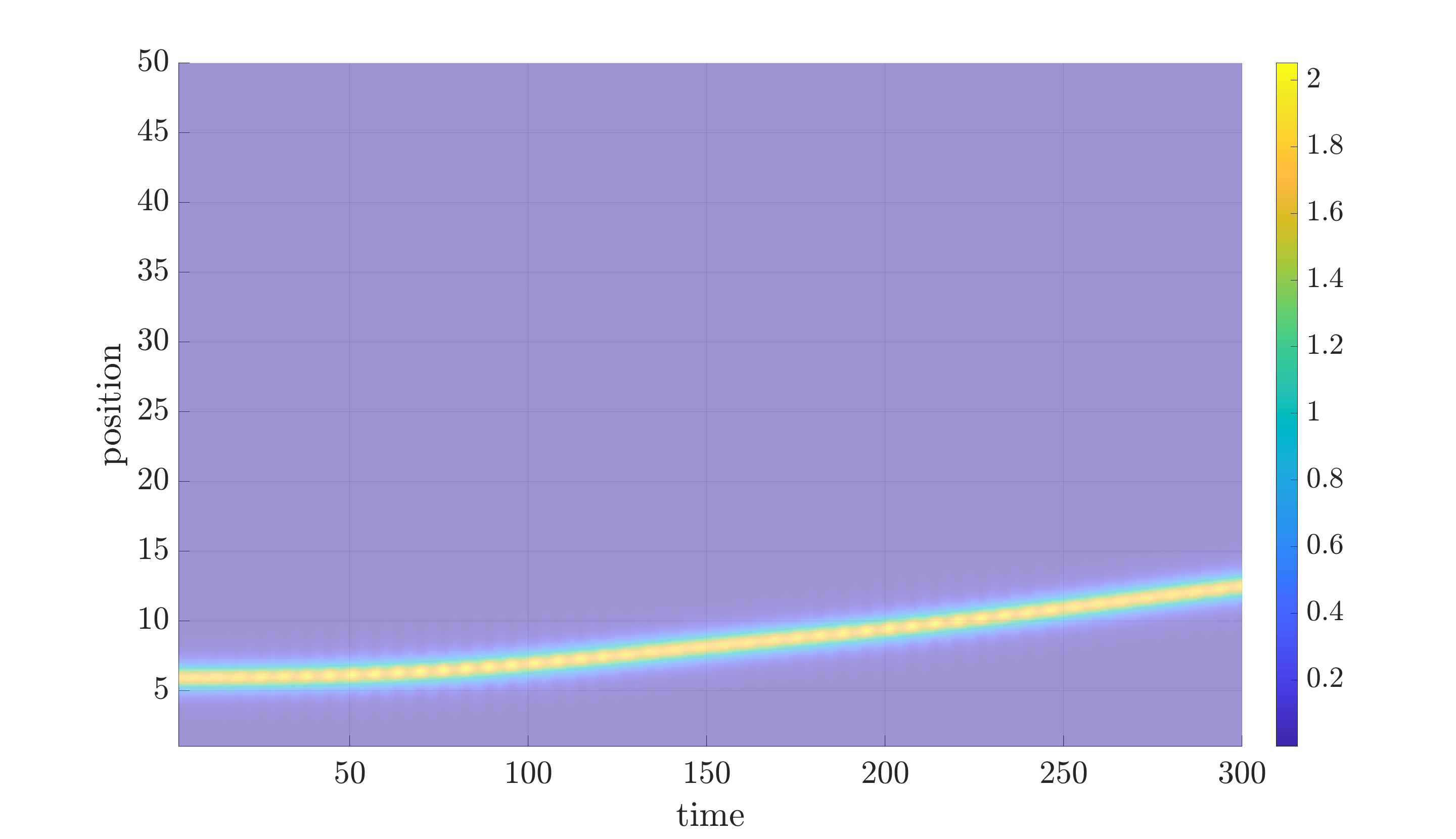}}  \quad
\subfloat[][\emph{$x_0=L/9$}]
{\includegraphics[width=.45\textwidth]{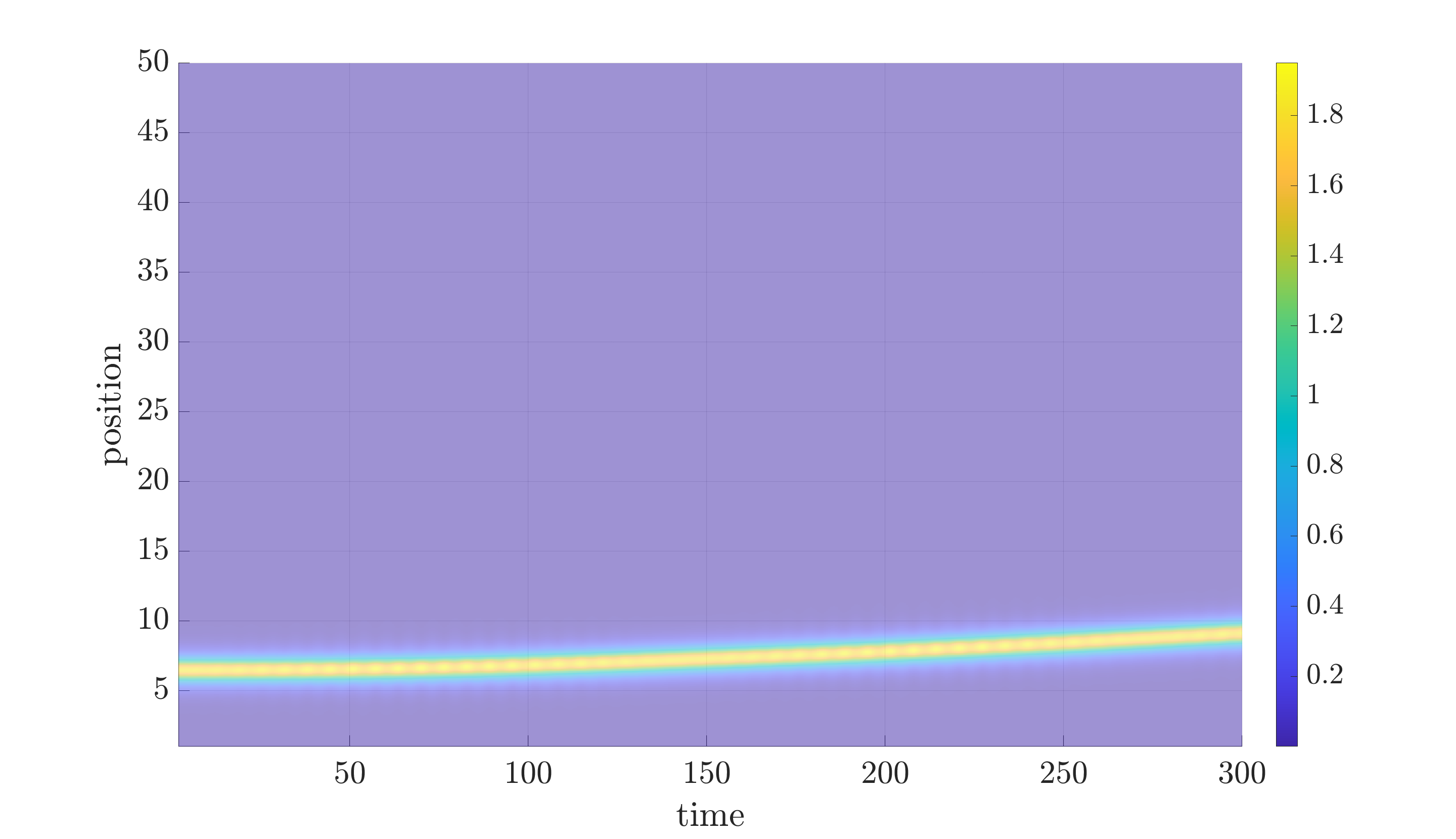}} \quad 
\subfloat[][\emph{$x_0=L/8$}]
{\includegraphics[width=.45\textwidth]{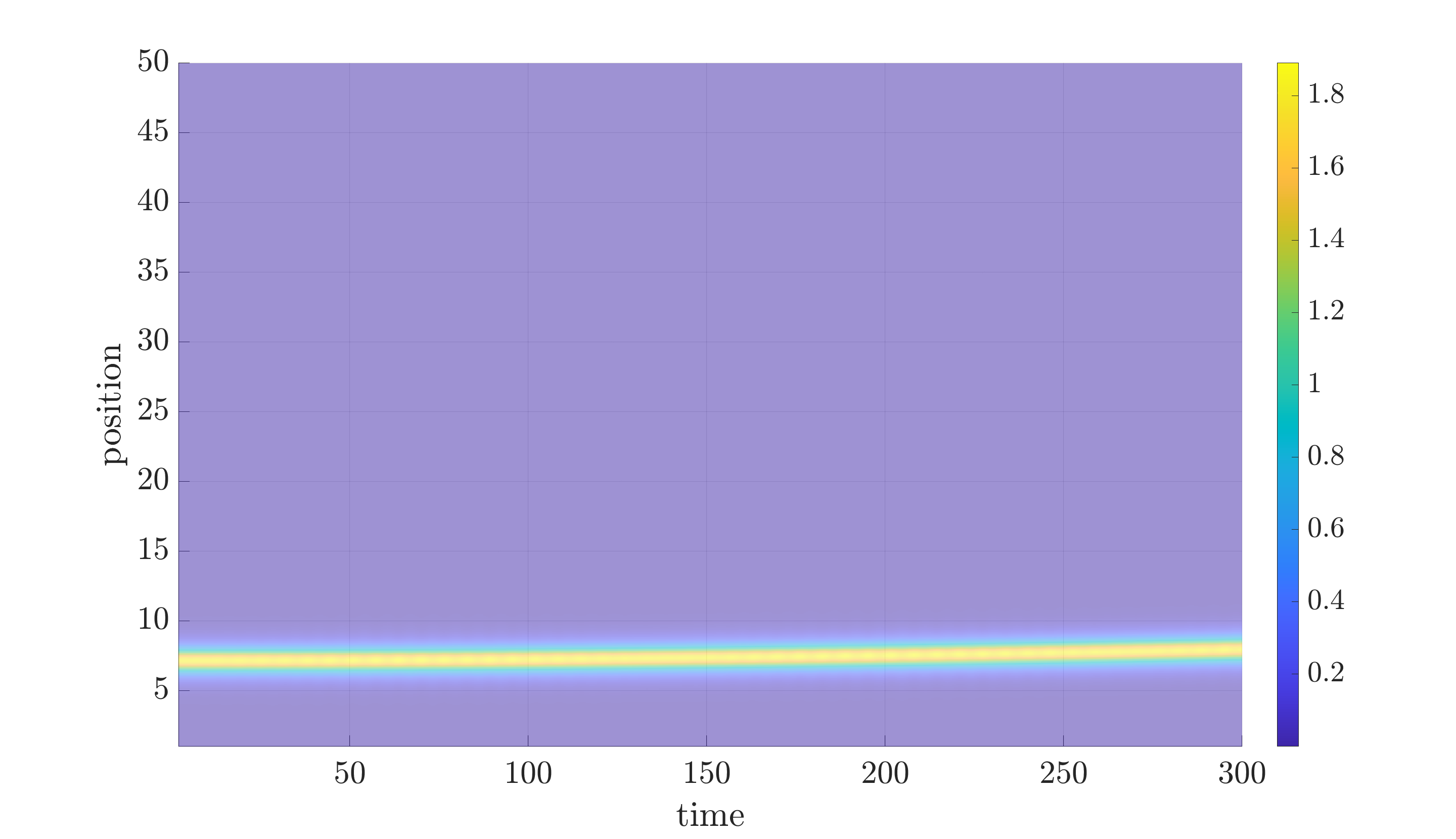}} \quad
\subfloat[][\emph{$x_0=L/2$}]
{\includegraphics[width=.45\textwidth]{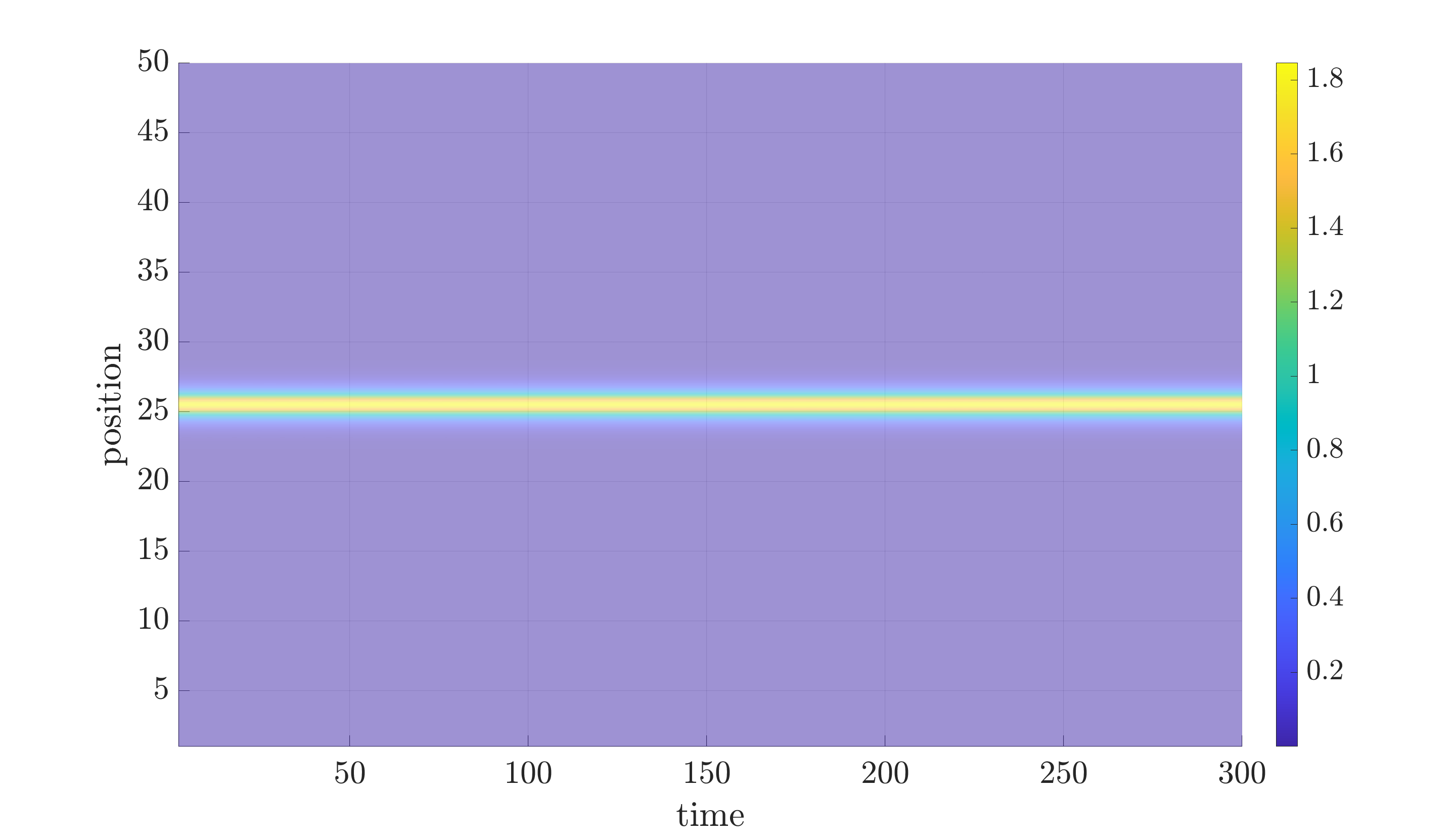}} \\
\caption{We plot  the amplitude $|u(x,t)|^2$ of the solution to the NLS equation \eqref{NLS_eq} with $p=5$, as a function of space and time. The initial datum is a soliton which is located on a single edge of the graph at distance $x_0$ from the vertex and with zero incoming velocity. Only the edge where the soliton is confined is displayed. As before, the position of the vertex corresponds to the origin of the vertical axis. The location of the initial datum is varied from $x_0=L/10$ (the closest to the vertex) to $x_0=L/2$ (the furthest from the vertex). We observe how the repulsive effect diminishes as $x_0$ increases, i.e. as we move away from the compact part of the graph. }
\label{sub_graphs_2}
\end{figure}
\begin{figure}
\centering
\subfloat[][\emph{$x_0=0.875 L$}]
{\includegraphics[width=.45\textwidth]{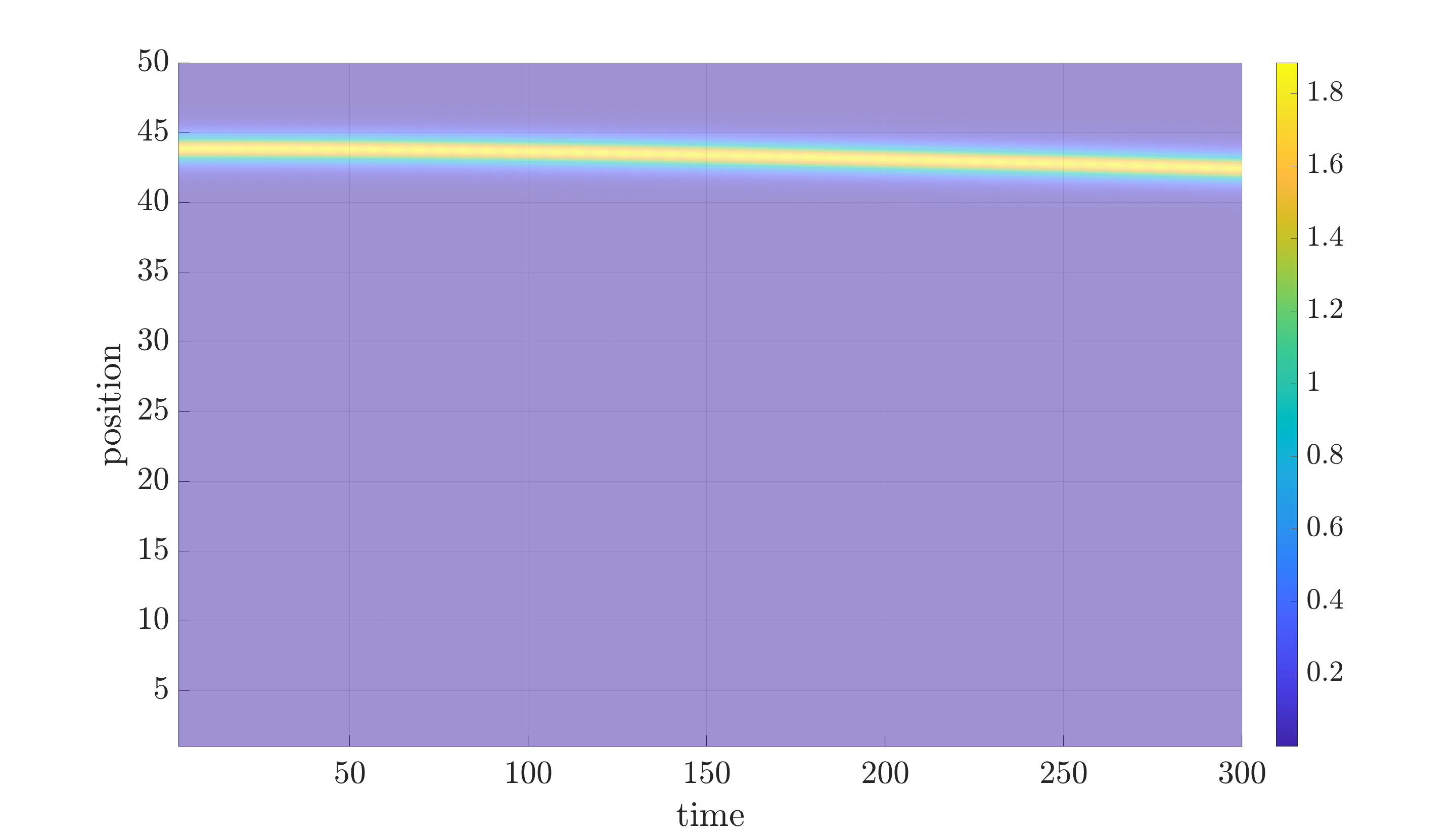}} \quad
\subfloat[][\emph{$x_0=0.9 L$}]
{\includegraphics[width=.45\textwidth]{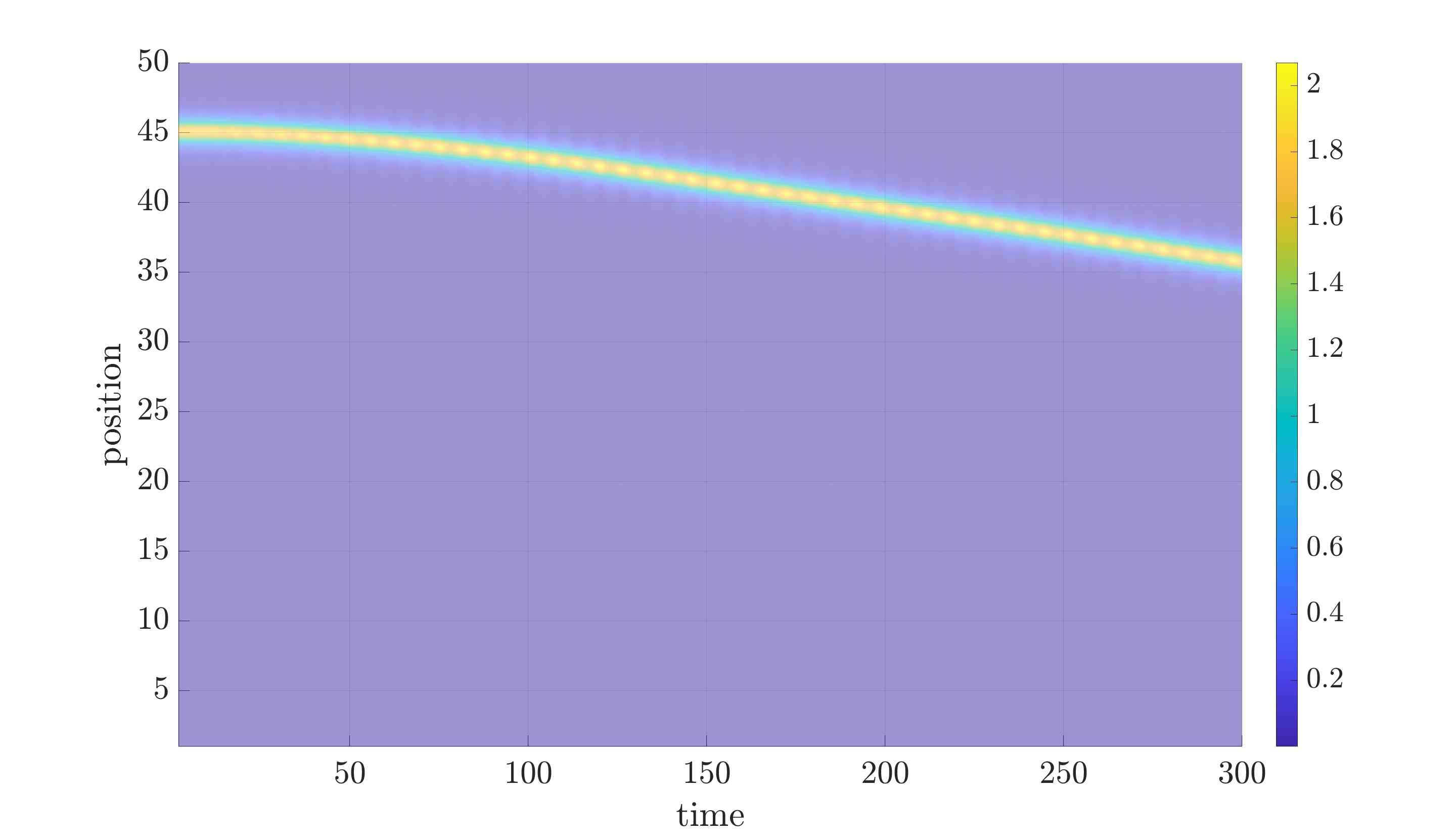}} \\
\caption{Similar plot as in Figure \ref{sub_graphs_2}. In (A) the initial datum is centered in $x_0 = 0.875L$ and in (B) it is centered in $x_0=0.9L$. In both cases we observe the presence of boundary effects, which push the soliton away from the free end of the edge. }
\label{sub_graphs_3}
\end{figure}
\begin{figure}
\centering
\subfloat[][\emph{$x_0=0.875 L$ \ with \ $v=-0.1$}]
{\includegraphics[width=\textwidth]{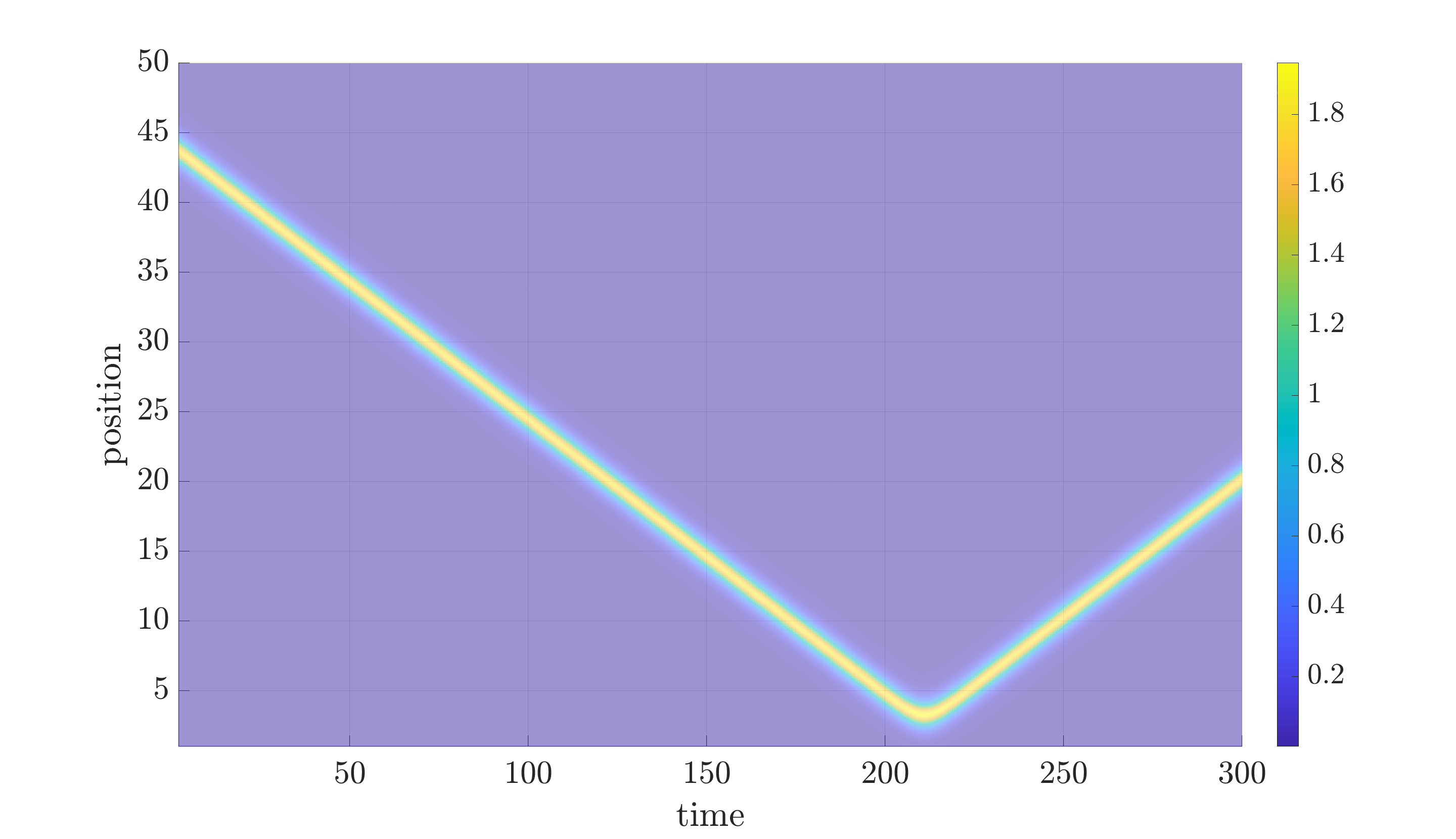}} \\
\caption{Similar plot as in Figure \ref{sub_graphs_2}. This time the initial datum corresponds to a soliton located at distance $x_0 = 0.875L$ from the vertex and with incoming velocity $v=-0.1$. We observe that the soliton proceeds in uniform motion towards the vertex of the graph. It is then reflected by the collision with the vertex. In this case, the presence of boundary effects is less apparent.}
\label{sub_graphs_4}
\end{figure}

\end{document}